\documentclass[12pt,twoside]{amsart}
\usepackage[margin=3cm]{geometry}
\usepackage[colorlinks=false]{hyperref}
\usepackage[english]{babel}
\usepackage{graphicx,titling}
\usepackage{float}
\usepackage{amsmath,amsfonts,amssymb,amsthm}
\usepackage{lipsum}
\usepackage[T1]{fontenc}
\usepackage{fourier}
\usepackage{color}
\usepackage[latin1]{inputenc}
\usepackage{esint}
\usepackage{caption}
\usepackage{ccicons}

\makeatletter
\def\blfootnote{\xdef\@thefnmark{}\@footnotetext}
\makeatother

\newcommand\ccnote{
    \blfootnote{\copyright\,\, Guido De Philippis and Alessandro Pigati}
    \blfootnote{\ccLogo\, \ccAttribution\,\, Licensed under a \href{https://creativecommons.org/licenses/by/4.0/}{Creative Commons Attribution License (CC-BY)}.}
}

\usepackage[export]{adjustbox}
\numberwithin{equation}{section}
\usepackage{setspace}
\renewcommand{\le}{\leqslant}

\renewcommand{\ge}{\geqslant}

\renewcommand{\mathbb}{\varmathbb}
\usepackage{fancyhdr}
\newtheorem{theorem}{Theorem}[section]
\newtheorem{lemma}[theorem]{Lemma}
\newtheorem{corollary}[theorem]{Corollary}
\newtheorem{proposition}[theorem]{Proposition}
\newtheorem{definition}[theorem]{Definition}
\newtheorem{remark}[theorem]{Remark}
\usepackage{xcolor}
\usepackage{bm}
\usepackage{enumitem}
\usepackage{mathtools}
\usepackage{stmaryrd}
\usepackage{overpic}
\usepackage{comment}
\usepackage[capitalise,nameinlink]{cleveref}
\crefformat{equation}{#2(#1)#3}
\crefrangeformat{equation}{#3(#1)#4 to~#5(#2)#6}
\crefmultiformat{equation}{#2(#1)#3}{ and~#2(#1)#3}{, #2(#1)#3}{ and~#2(#1)#3}
\AddToHook{env/lemma/begin}{\crefalias{theorem}{lemma}}
\AddToHook{env/proposition/begin}{\crefalias{theorem}{proposition}}
\AddToHook{env/corollary/begin}{\crefalias{theorem}{corollary}}
\AddToHook{env/definition/begin}{\crefalias{theorem}{definition}}
\AddToHook{env/remark/begin}{\crefalias{theorem}{remark}}
\AddToHook{env/conj/begin}{\crefalias{theorem}{conj}}
\crefname{conj}{Question}{Questions}

\newtheorem{conj}[theorem]{Question}

\makeatletter
\newcommand{\eqlabel}[1]{\refstepcounter{equation}\textup{\tagform@{\theequation}}\label{#1}}
\makeatother
\newcommand{\N}{\mathbb{N}}
\newcommand{\Z}{\mathbb{Z}}

\newcommand{\R}{\mathbb{R}}

\newcommand{\de}{\partial}

\newcommand{\mz}{\frac{1}{2}}
\newcommand{\ang}[1]{\langle#1\rangle}
\newcommand{\uno}{\bm{1}}
\newcommand{\nin}{\not\in}
\newcommand{\mrestr}{\mathbin{\vrule height 1.6ex depth 0pt width
0.13ex\vrule height 0.13ex depth 0pt width 1.3ex}}

\renewcommand{\bar}{\overline}
\renewcommand{\phi}{\varphi}

\DeclareMathOperator{\dive}{div}
\newcommand{\reg}{\mathcal{R}}
\newcommand{\reh}{\mathcal{S}}
\newcommand{\str}{\mathcal{U}}
\newcommand{\hau}{\mathcal{H}}
\newcommand{\A}{\mathcal{A}}
\renewcommand{\epsilon}{\varepsilon}
\renewcommand{\hat}{\widehat}
\newcommand{\weakto}{\rightharpoonup}
\newcommand{\gr}{\text{Gr}_k(\R^n)}

\address{Guido De Philippis, Department of Mathematics, University of Padova, Via Trieste 63, 35121 Padova, Italy}
\email{guido.dephilippis@math.unipd.it}
\medskip
\address{Alessandro Pigati, Department of Decision Sciences, Bocconi University, Via Roentgen 1, 20136 Milano, Italy}
\email{alessandro.pigati@unibocconi.it}

\begin{document}

\thispagestyle{empty}

\begin{minipage}{0.28\textwidth}
\begin{figure}[H]
%\centering
\includegraphics[width=2.5cm,height=2.5cm,left]{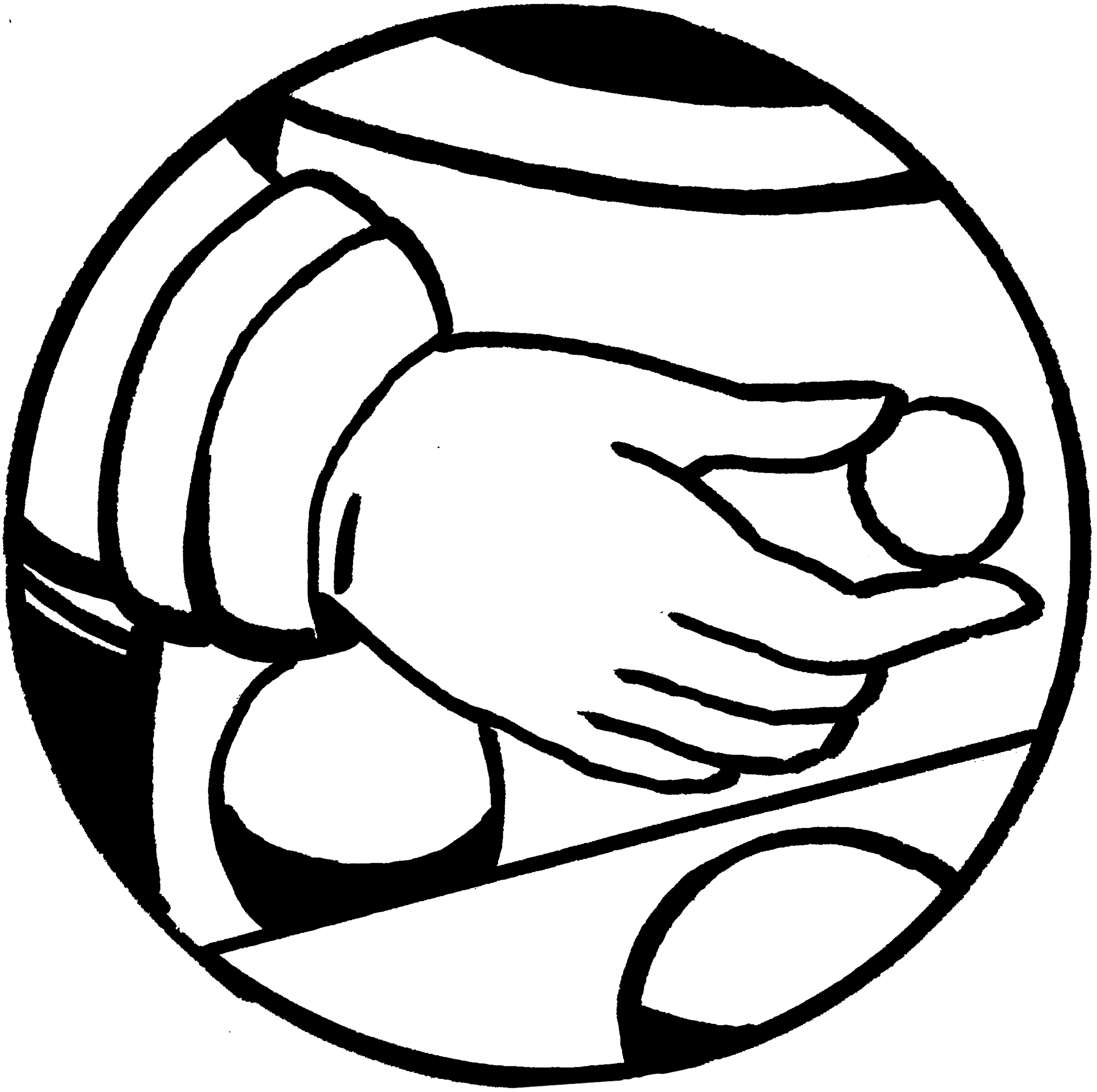}
\end{figure}
\end{minipage}
\begin{minipage}{0.7\textwidth}
\begin{flushright}
%% The following metadata, in particular
%% the Paper No. and the DOI will be inserted by the journal
Ars Inveniendi Analytica (2026), Paper No. 3, 33 pp.
\\
DOI 10.15781/g2t81073
\\
ISSN: 2769-8505
\end{flushright}
\end{minipage}

\ccnote

\vspace{1cm}

%%      -------------------------------------------------------------------------------
%%      -------------------------- TITLE ----------------------------
%%      -------------------------------------------------------------------------------
%% Authors, please put here the full title of the article

\begin{center}
\begin{huge}
\textit{Michael--Simon inequality for anisotropic}

\textit{energies close to the area via}

\textit{multilinear Kakeya-type bounds}

\end{huge}
\end{center}

\vspace{1cm}

%%      -------------------------------------------------------------------------------
%%      -------------------------- AUTHORS AND AFFILIATIONS ----------------------------
%%      -------------------------------------------------------------------------------
%% Authors, please put here your full names and affiliations

\hfill
\begin{minipage}[t]{.35\textwidth}
\begin{center}
{\large{\bf{Guido De Philippis}}} \\
\vskip0.15cm
\footnotesize{University of Padova}
\end{center}
\end{minipage}
\hfill
\begin{minipage}[t]{.35\textwidth}
\begin{center}
{\large{\bf{Alessandro Pigati}}} \\
\vskip0.15cm
\footnotesize{Bocconi University}
\end{center}
\end{minipage}
\hfill\null

\vspace{1cm}

\begin{center}
\noindent \em{Communicated by Francesco Maggi}
\end{center}
\vspace{1cm}

%%      -------------------------------------------------------------------------------
%%      -------------------------- BEGIN ABSTRACT ----------------------------
%%      -------------------------------------------------------------------------------
%% Authors, please put here the ABSTRACT and KEYBOARDS

\noindent \textbf{Abstract.} {\itshape Given an anisotropic integrand \(F:\gr\to(0,\infty)\), we can generalize the classical isotropic area by looking at the functional
\[
\mathcal{F}(\Sigma^k):=\int_\Sigma F(T_x\Sigma)\,d\mathcal{H}^k.
\]
While a monotonicity formula is not available for critical points~\cite{All.no.mono},
when \(k=2\) and \(n=3\) we show that the Michael--Simon inequality holds
if \(F\) is close to \(1\) (in \(C^1\)), meaning that \(\mathcal{F}\) is close to the usual area.

Our argument is partly based on some key ideas of Almgren, who proved this
result in an unpublished manuscript, but we largely simplify his original proof
by showing a new functional inequality for vector fields on the plane,
which can be seen as a quantitative version of Alberti's rank-one theorem,
and we also remove a convexity assumption on \(F\).

As another byproduct, we also show Michael--Simon for another class of integrands
which includes the \(\ell^p\) norms for \(p\in(1,\infty)\).
For a general \(F\) satisfying the atomic condition~\cite{DDG}, we also show that
the validity of Michael--Simon is equivalent to compactness of rectifiable varifolds.\par}
\vskip0.3cm

\vspace{0.5cm}

%%      -------------------------------------------------------------------------------
%%      -------------------------- BEGIN ARTICLE ----------------------------
%%      -------------------------------------------------------------------------------
%% Authors, copy the body of your paper here

\section{Introduction}

\subsection{Setting and main results}
Geometric measure theory and, more broadly, a large part of calculus of variations and geometric analysis
deal with the classical isotropic area. The study of its critical points \(\Sigma^k\subset\R^n\) makes extensive use of the \emph{monotonicity formula}, asserting that for all \(p\in\R^n\) we have
\[\frac{\mathcal{H}^k(\Sigma\cap B_r(p))}{r^k}\le\frac{\mathcal{H}^k(\Sigma\cap B_s(p))}{s^k}\quad\text{for }0<r<s.\]
 This fact has a number of useful consequences, best phrased in terms of varifolds (see, e.g.,~\cite[Chapters 4 and 8]{Simon}): among the fundamental ones, existence and
upper semicontinuity of the density for stationary varifolds, upper semicontinuity of their support
under varifold convergence,
compactness of rectifiable and integral varifolds (either stationary or with local uniform bounds on the first variation), and existence and conical symmetry of blow-ups, the latter following
from a more precise version of monotonicity.

There is a natural anisotropic generalization of the area functional, given by
taking an integrand \(F:\gr\to(0,\infty)\) of class \(C^1\) and defining
\[\mathcal{F}(\Sigma):=\int_\Sigma F(T_x\Sigma)\,d\mathcal{H}^k(x)\]
for a smoothly embedded \(\Sigma^k\subset\R^n\), extending the definition to \(k\)-varifolds in the obvious way.
In the anisotropic setting, it is known~\cite{All.no.mono} that monotonicity fails
(in the sense that if \(\mathcal{F}\) satisfies an identity resembling too closely
the quantitative version of monotonicity, then \(\mathcal{F}\) is the isotropic area up to a linear change of coordinates).
Thus, many basic tools break down at this level of generality.

However, a weaker (though arguably more robust) fact is believed to hold
for appropriate classes of integrands \(F\).
Specifically, for \(n\ge3\) and \(k\in \{2,\dots,n-1\}\), the \emph{Michael--Simon inequality}
(first proved in~\cite{MS} for the isotropic area, as a consequence of monotonicity)
is conjectured to hold for appropriate \(F\), leading to the following question.

\begin{conj}\label{ms.conj}
    For which \(F\) does it hold that,
    given a rectifiable \(k\)-varifold \(V\) in \(\R^n\) with \(\Theta^k(|V|,x)\ge\theta_0>0\) for \(|V|\)-a.e.\ \(x\), as well as finite total mass and first variation:
    \[
     |V|(\R^n),\quad |\delta^FV|(\R^n)<\infty,
    \]
    we have
    \begin{equation}\label{ms}
     |V|(\R^n)^{k-1}\le \frac{C(n,k,F)}{\theta_0}|\delta^FV|(\R^n)^k\quad\text{?}
    \end{equation}
\end{conj}

While the range of applications of this bound would be far more limited compared to a monotonicity formula (e.g., it would not give upper density bounds or conicality of blow-ups),
it would still have a number of key consequences, such as compactness of rectifiable and integral varifolds and, when the first variation is in \(L^p\) with \(p>k\), a lower density bound of the form
\[|V|(B_r(p))\ge cr^k\quad\text{for }p\in\operatorname{spt}(|V|),\ r<1\]
(see~\cref{lp.lb} below) and upper semicontinuity of the support along converging sequences of varifolds.

\begin{remark}
It is clear that, by scaling and normalization of density,~\cref{ms} is equivalent to its validity
when \(\theta_0=1\) and \(|V|(\R^n)=1\). Moreover, it is equivalent to the \emph{functional} version of Michael--Simon (see~\cref{functional} below):
for any \(f\in C^1_c(\R^n)\) we have
\begin{equation}\label{ms.f}
    \Big[\int_{\R^n}|f|^{k/(k-1)}\,d|V|\Big]^{(k-1)/k}\le \frac{C'(n,k,F)}{\theta_0^{1/k}}
\Big[\int_{\R^n}|df|\,d|V|+\int_{\R^n}|f|\,d|\delta^FV|\Big].
\end{equation}
\end{remark}

In codimension one (when \(k=n-1\)), conjecturally one sufficient assumption should be \emph{strict convexity} of \(F\), by virtue of the fact that it is equivalent to the atomic condition (AC), and in turn to rectifiability under the assumption \(\Theta^k_*(|V|,\cdot)>0\) a.e.~\cite{DDG}.
More precisely,
once we identify \(\text{Gr}_{n-1}(\R^n)\) with \(\mathbb{RP}^{n-1}\) (i.e., a hyperplane with its
unit normal \(\pm\nu\)) and \(F:\text{Gr}_{n-1}(\R^n)\to(0,\infty)\) with an even function
\(F:\mathbb{S}^{n-1}\to(0,\infty)\), we extend the latter to a 1-homogeneous function \(F:\R^n\to[0,\infty)\),
which is then required to be convex, and actually strictly convex along lines not passing through \(0\).

One of the main results of the present paper is to answer~\cref{ms.conj} affirmatively
for surfaces in \(\R^3\), when \(F\) is close enough to the area. This result was first
proved in an unpublished manuscript of Almgren, from which we borrowed several key ideas,
although we bypass a number of technical steps from his proof by leveraging some new functional inequalities presented below. Interestingly, as opposed to his, our proof does not require convexity of \(F\).

\begin{theorem}\label{main}
    If \(n=3\), \(k=2\), and \(F\) is an integrand close enough to the isotropic area in the \(C^1\) topology (i.e., \(\|F|_{\mathbb{S}^{2}}-1\|_{C^1}\) is small enough), then the following holds.
    Given a rectifiable \(2\)-varifold \(V\) in \(\R^3\) with finite total mass and first variation,
    letting \(\theta(x):=\Theta_*^2(|V|,x)\), we have the scale-invariant bound
	\begin{align}\label{ms.stronger}
		&|V|(\R^3)\le C(F)\mathcal{H}^2(\{\theta>0\})^{1/2}\cdot|\delta^F V|(\R^3).
	\end{align}
	In particular, if \(\theta\ge\theta_0\) \(|V|\)-a.e.,~\cref{ms} holds for a possibly different constant \(C(F)\).
\end{theorem}

Note that the second conclusion follows from the first one, since we can bound the measure \(\mathcal{H}^2(\{\theta>0\})=\mathcal{H}^2(\{\theta\ge\theta_0\})\) from above by \(\theta_0^{-1}|V|(\R^3)\).
It should be noted that the proof is not perturbative: in fact, it proceeds by
singling out an explicit set of integrands \(F\) (open in the \(C^1\) topology for \(F|_{\mathbb{S}^{2}}\))
which happens to contain the isotropic area. %
This result immediately implies the first assumption in the regularity lemma from~\cite[p.\ 25]{All.const}.
Hence, we can apply Allard's regularity theorem in the anisotropic setting~\cite[p.\ 27]{All.const}
and obtain the following, where (unlike in the other statements) strict convexity is understood in the quantitative sense, namely as
\(\operatorname{Hess}F(\nu)\ge\lambda I\) on each plane \(\nu^\perp\), for some constant \(\lambda>0\).

\begin{theorem}\label{allard.reg}
    If \(n=3\), \(k=2\), and \(F\) is a strictly convex smooth integrand close enough to the isotropic area in the \(C^1\) topology,
    then the following holds for some \(\epsilon(F,\alpha)>0\). Given an integer-rectifiable \(2\)-varifold \(V\) in \(B_{2r}(x_0)\subset\R^3\) and \(x_0\in\operatorname{spt}(|V|)\),
    if for some \(\nu_0\in\mathbb{S}^2\) we have
    \[\frac{|V|(B_{2r}(x_0))}{\pi(2r)^2}\in\Big[\frac12,\frac32\Big],\]
    as well as
    \[|\delta^FV|\le\Lambda|V|\]
    for some \(\Lambda\in(0,\varepsilon/r)\), and
    \[\int_{B_{2r}(x_0)}\ang{x-x_0,\nu_0}^2\,d|V|(x)\le\epsilon r^4,\]
    then in the ball \(B_r(x_0)\) the varifold \(V\) agrees with a \(C^{1,\alpha}\) graph of multiplicity \(1\)
    over the plane \(\nu_0^\perp\),
    for any given \(\alpha\in(0,1)\) (with a scale-invariant \(C^{1,\alpha}\) bound, vanishing as \(\epsilon\to0\)).
   \end{theorem}

Moreover, for the same dimensions, using similar ideas (but a simpler tool, namely~\cref{det.simple.intro}
in place of~\cref{main.tw.intro}), we also obtain the following. Note that
the technical condition in its statement holds for any \(\ell^p\) norm with \(p\in(1,\infty)\),
namely if
\[F(\nu)=(|\nu^x|^p+|\nu^y|^p+|\nu^z|^p)^{1/p}.\]

\begin{theorem}\label{main.tris}
    Assume that \(n=3\), \(k=2\), and \(F:\R^3\to[0,\infty)\) is a strictly convex integrand, even in each component.
    Moreover, for any coordinate plane \(P\), assume that
    \(\frac{\pi_P(\nabla F(\nu))}{|\pi_P(\nabla F(\nu))|}\) depends only on \(\frac{\pi_P(\nu)}{|\pi_P(\nu)|}\),
    for all \(\nu\in\mathbb{S}^2\setminus\pi_P^{-1}(0)\) (note that
    \(\pi_P(\nabla F(\nu))\neq0\) for such \(\nu\)).
    Then the bound~\cref{ms.stronger} holds true.
\end{theorem}

\subsection{New variants of multilinear Kakeya in dimension two}
Almgren's argument is quite technical in that it
involves studying the lengths of carefully chosen pieces of curves (corresponding to vector fields
with summable divergence appearing as rows of the first variation matrix),
consisting of points where a certain projected density is large enough
and such vector fields are transverse enough.
Our argument borrows some of his key ideas, such as~\cref{g.def}, and bakes them
directly into a new standalone functional inequality (see~\cref{kak.bis.intro}), from which~\cref{main} follows quite directly.

Before stating this key inequality, let us mention that one of the starting points
in its proof was the following simpler version.

\begin{theorem}\label{det.simple.intro}
    Given \(S,T\in W^{1,2}(\R^2,\R^2)\), assume that \(S^x,T^y\ge 0\) and \(\det(S,T)\ge 0\) a.e. Then
	\begin{align*}
		&\int_{\R^2}\det(S,T)\le\frac{1}{4}\Big(\int_{\R^2}\lvert \dive S\rvert \Big)\Big(\int_{\R^2}\lvert \dive T\rvert\Big).
    \end{align*}
\end{theorem}

In fact, this result was already obtained in~\cite{GRS}, where other nonlinear constraints are studied.
We will present nonetheless a short proof of it because, besides recovering the sharp constant,
our proof uses (in spirit) Smirnov's decomposition of normal \(1\)-currents as superpositions of curves~\cite{Smirnov}, which may be thought of as the intuitive reason why this holds (as discussed in~\cref{sec.nonlinear}), and the same technique can be used to prove the following instrumental bound,
where crucially we drop the assumption \(\det(S,T)\ge0\).

\begin{theorem}\label{main.tw.intro}
	Given two vector fields \(S,T\in W^{1,1}(\R^2,\R^2)\), we define
	\begin{align*}
		\begin{aligned}
		S^P&:=(S^x-|S^y|)^+,\quad S^N:=(S^x-|S^y|)^-, \\
		T^P&:=(T^y-|T^x|)^+,\quad T^N:=(T^y-|T^x|)^-.
		\end{aligned}
	\end{align*}
    Also, let \(\chi:\R^2\to[0,1]\) be a Borel function supported in a bounded set.
    Then, for some universal constant \(C>0\), we have
	\begin{align}\label{kak.bis.intro}
        \begin{aligned}
		\int_{\R^2} \chi \min\{S^P,T^P\}
		&\le C\|\chi\|_{L^2}\Big[\int_{\R^2}(|S|+\lvert \dive S\rvert)\Big]^{1/2}\Big[\int_{\R^2}(|T|+\lvert \dive T\rvert)\Big]^{1/2} \\
		&\quad+C\int_{\R^2} (\lvert \dive S\rvert+\lvert \dive T\rvert)
        +\int_{\R^2}(S^N+T^N).
        \end{aligned}
	\end{align}
    The same holds if \(S^x,S^y,T^x,T^y,\dive S,\dive T\) are just real-valued measures on the plane with finite total variation.
\end{theorem}

Besides its own interest, this inequality can be viewed as a quantitative
version of existing results, such as~\cite[Proposition 8.6]{ACP} or Alberti's rank-one theorem,
proved initially in~\cite{Alberti} and re-obtained with different techniques in~\cite{DR,MV}.

\begin{corollary}
    Given \(u\in BV(\R^n,\R^m)\) and writing \((Du)^s=A|Du|^s\) (polar decomposition of the singular part of \(Du\)), we have \(\operatorname{rk}(A)=1\)
    at \(|Du|^s\)-a.e.\ point.
\end{corollary}

We refer to~\cref{sec.nonlinear} for the short argument used to deduce this from~\cref{main.tw.intro}, giving yet another proof of Alberti's rank-one theorem.

\subsection{General results in arbitrary dimension}
Back to the general setting of arbitrary \(k\), \(n\), and \(F\), in~\cite{DDG} the \emph{atomic condition (AC)} for \(F\) was introduced. Namely, \(F\) satisfies (AC) if any average
\[\int_{\gr}B_F(P)\,d\lambda(P),\quad\lambda\in\mathcal{P}(\gr)\]
has rank \(\ge k\), with equality if and only if \(\lambda=\delta_{P_0}\) is a Dirac mass;
here \(B_F(P)\) is the matrix naturally associated with \(P\in\gr\) in the computation of the first variation
(see~\cref{bf}).
It was shown in~\cite{DDG} that (AC) holds true
if and only if, for any varifold \(V\) with locally bounded first variation,
the condition \(\Theta^{k}_*(|V|,x)>0\) for \(|V|\)-a.e.\ \(x\) implies (and hence is equivalent to) the rectifiability of \(V\).

Thus, it constitutes a very natural assumption that we will make throughout the rest of this introduction
(specifically, in~\cref{equiv} and~\cref{cpt.int}).
In this paper, we will also prove the following general facts.

\begin{proposition}\label{equiv}
    Given \(n\ge3\), \(k\in\{2,\dots,n-1\}\), and \(F:\gr\to(0,\infty)\) satisfying (AC), the following are equivalent.
    \begin{itemize}
        \item[(i)] The Michael--Simon bound~\cref{ms} holds true for some \(C(n,k,F)>0\).
        \item[(ii)] We have compactness of rectifiable varifolds: given a sequence \((V_i)_{i\in\N}\) of rectifiable \(k\)-varifolds, in \(\R^n\) or (equivalently) in the torus \(\mathbb{T}^n=\R^n/\Z^n\), if
        \[\Theta^k(|V_i|,x)\ge\theta_0>0\text{ for \(|V_i|\)-a.e.\ \(x\)},\quad\sup_{i\in\N}|\delta^FV_i|(K)<\infty\]
        for any compact set \(K\), and \(V_i\weakto V\), then \(V\) is a rectifiable \(k\)-varifold
        with \(\Theta^k(|V|,x)\ge\theta_0>0\) for \(|V|\)-a.e.\ \(x\).
        \item[(iii)] There is no sequence \((V_i)_{i\in\N}\) of rectifiable \(k\)-varifolds in
        \(\mathbb{T}^n\) such that
        \[|V_i|(\mathbb{T}^n)=1,\quad\Theta^k(|V_i|,x)\ge i\text{ for \(|V_i|\)-a.e.\ \(x\)},\quad|\delta^FV_i|(\mathbb{T}^n)\to0\]
        and \(V_i\weakto\mathcal{L}^n\otimes\lambda\) as measures on \(\mathbb{T}^n\times\gr\),
        for some probability measure \(\lambda\in\mathcal{P}(\gr)\).
        \item[(iv)]
         There is no sequence \((V_i)_{i\in\N}\) of rectifiable \(k\)-varifolds in \(\R^n\), with
         \(\operatorname{spt}(|V_i|)\subseteq[0,1]^n\), such that
        \[|V_i|(\mathbb{R}^n)=1,\quad\Theta^k(|V_i|,x)\ge i\text{ for \(|V_i|\)-a.e.\ \(x\)},\quad\sup_{i\in\N}|\delta^FV_i|(\mathbb{R}^n)<\infty\]
        and \(V_i\weakto(\mathcal{L}^n\mrestr[0,1]^n)\otimes\lambda\) as measures on \(\mathbb{R}^n\times\gr\),
        for some probability measure \(\lambda\in\mathcal{P}(\gr)\).
    \end{itemize}
    In fact, in (ii), the fact that \(\Theta^k(|V|,x)\ge\theta_0>0\) for \(|V|\)-a.e.\ \(x\)
    holds automatically if \(V\) is rectifiable.
\end{proposition}

Hence, if compactness of rectifiable varifolds (or, equivalently, Michael--Simon) fails,
then we can find a counterexample exhibiting a phenomenon called \emph{diffuse concentration}:
the measures \(|V_i|\) are concentrated on Borel sets \(E_i\) with \(\mathcal{H}^k(E_i)\to0\)
(and thus their projection \(\pi_P(E_i)\) on any coordinate \(k\)-plane has \(\mathcal{L}^k(\pi_P(E_i))\to0\),
making \((\pi_P)_*|V_i|\) look more and more like a singular measure), but nonetheless
their limit \(|V|\) happens to be the Lebesgue measure (whose projection \((\pi_P)_*|V|\) is Lebesgue).

Moreover, by an adaptation of Allard's strong constancy lemma~\cite[p.\ 3]{All.const} (see also~\cite{DDG,DR}), we can show the following statement, in which \(S(\mathcal{D})\) denotes the finite-sum set of \(\mathcal{D}\subseteq\R\), i.e., the set of all possible finite, nonempty sums of elements of \(\mathcal{D}\) (possibly with repetitions).

\begin{proposition}\label{sumset}
    Given a sequence of rectifiable \(k\)-varifolds, in \(\R^n\) or in \(\mathbb{T}^n=\R^n/\Z^n\),
    converging to a rectifiable \(k\)-varifold \(V\), suppose that for each \(i\in\N\) we have a set \(\mathcal{D}_i\subseteq(0,\infty)\) with \(\inf\mathcal{D}_i>0\) and a Borel set \(E_i\subseteq\R^n\) (or $E_i\subseteq\mathbb{T}^n$) such that
    \[\Theta^k(|V_i|,x)\in \mathcal{D}_i\text{ for \(|V_i|\)-a.e.\ \(x\nin E_i\)},\]
    as well as \(|V_i|(E_i\cap K)\to0\) and \(\sup_i|\delta^FV_i|(K)<\infty\)
    for any compact set \(K\). Then, up to a subsequence, for \(|V|\)-a.e.\ \(x\) we have
    \[\Theta^k(|V|,x)=\lim_{i\to\infty}v_i\]
    for suitable \(v_i\in S(\mathcal{D}_i)\) depending on \(x\).
\end{proposition}

The following is a direct consequence of the previous two facts.

\begin{corollary}\label{cpt.int}
    Given \(F\) satisfying (AC),
    if any of the equivalent conditions above holds, then we also have compactness of integer-rectifiable varifolds:
    given a sequence \((V_i)_{i\in\N}\) of rectifiable \(k\)-varifolds, in \(\R^n\) or in \(\mathbb{T}^n\), if
    \[\Theta^k(|V_i|,x)\in\N\setminus\{0\}\text{ for \(|V_i|\)-a.e.\ \(x\)},\quad\sup_{i\in\N}|\delta^FV_i|(K)<\infty\]
    for any compact set \(K\), and \(V_i\weakto V\), then \(V\) is a rectifiable \(k\)-varifold
    with \(\Theta^k(|V|,x)\in\N\setminus\{0\}\) for \(|V|\)-a.e.\ \(x\).
\end{corollary}

Let us now state another simple consequence of~\cref{ms}. Recall that, given \(p\in[1,\infty]\) and a \(k\)-varifold \(V\), we say that its first variation is \emph{locally in \(L^p\)} if \(V\) has locally bounded first variation and \(|\delta^FV|=f|V|\) for some \(f\in L^p_{loc}(|V|)\).

\begin{corollary}\label{usc}
    If any of the equivalent conditions of~\cref{equiv} holds, then the following hold true as well.
    \begin{itemize}
        \item[(i)]
    Given a rectifiable \(k\)-varifold \(V\), assume that \(\Theta^k(|V|,x)\ge\theta_0\) for \(|V|\)-a.e.\ \(x\) and that \(\delta^FV\) is locally in \(L^p\) for some \(p\in (k,\infty]\).
    Then, for all \(x\in\operatorname{spt}(|V|)\), we have \[\Theta^k_*(|V|,x)\ge c(n,k,F)\theta_0.\]
    More quantitatively, if \(|\delta^FV|=f|V|\) with \(\|f\|_{L^p(B_1(x),|V|)}\le\Lambda\),
    then there exists a radius \(r_0(n,k,p,F,\theta_0^{-1/p}\Lambda)\in(0,1)\) such that
    \begin{equation}\label{lp.lb}|V|(B_r(x))\ge c(n,k,F)\theta_0r^k\quad\text{for all }r\in(0,r_0).\end{equation}
    \item[(ii)]
    Moreover, if \(V_i\weakto V\) are as in (i) and \(\limsup_{i\to\infty}\|f_i\|_{L^p(K,|V_i|)}<\infty\) (where \(|\delta^FV_i|=f_i|V_i|\)) for any compact set \(K\), then
    \[\operatorname{spt}(|V_i|)\to\operatorname{spt}(|V|)\]
    locally in the Hausdorff sense, namely, for every compact set \(K\) and every \(\epsilon>0\),
    for \(i\) large we have \(\operatorname{spt}(|V_i|)\cap K\subseteq B_\epsilon(\operatorname{spt}(|V|))\)
    and \(\operatorname{spt}(|V|)\cap K\subseteq B_\epsilon(\operatorname{spt}(|V_i|))\).
    \end{itemize}
\end{corollary}

\begin{remark}
    While we study autonomous integrands \(F(P)\), for simplicity and in order to have scale-invariant bounds when possible, in general one can consider
    non-autonomous ones, of the form \(F(x,P)\). This generalization is necessary
    in order to look at anisotropic integrands on manifolds \(M^n\), where \(F\) is
    defined on the Grassmannian bundle of \(k\)-planes and thus locally takes this form.
    Compared to the autonomous case, the general case just introduces an error term
    which is easily handled (letting \({\tilde\delta}^FV\) denote the right-hand side of~\cref{fv.bf}, we have \(|{\tilde\delta}^FV|\le|\delta^FV|+C|V|\) locally). Straightforward modifications show that all the stated consequences of
    Michael--Simon still hold true, assuming this bound in a local form such as
    \[|V|(\R^n)^{k-1}\le \frac{C(n,k,F,K)}{\theta_0}|\delta^FV|(\R^n)^k\quad\text{whenever }\operatorname{spt}(|V|)\subseteq K,\]
    where \(K\subset\R^n\) (or \(K\subseteq M^n\)) is an arbitrary compact set.
\end{remark}

\section*{Addendum}
After the first version of the manuscript appeared, there were several developments on the same topic.
First, the main result was extended to arbitrary codimension in~\cite{FT1}, using similar techniques.\footnote{The same result was also obtained in an unpublished collaboration with B. B\'ek\'esi and M. Freguglia.}

Then, in~\cite{FT2}, the Michael--Simon bound was deduced in arbitrary dimension and codimension, for integrands close to the isotropic area, as well as all convex ones in codimension one, as a consequence of the groundbreaking resolution of the \emph{vanishing mass conjecture}~\cite{GR}.
Finally, through a generalization of certain results from~\cite{GR} to PDE-constrained measures supported on rectifiable sets,~\cite{DGPR} settled the validity of Michael--Simon for the entire class of (AC1) integrands, which includes the class (AC).

These results were used in~\cite{DFT} in order to prove an analogue of~\cref{allard.reg} for a broad class of integrands in general dimension and codimension, namely those satisfying an additional ellipticity requirement called \emph{quadratic exposed condition (QEC)}.

\section*{Acknowledgements}
As already mentioned, this paper uses many ideas due to Frederick J.\ Almgren Jr.
The authors are very grateful to Jean Taylor, who shared with them the unpublished manuscript of Almgren. They also wish to thank William Cooperman and Adolfo Arroyo-Rabasa
for identifying an issue in an earlier draft, which prompted a more general formulation of~\cref{main.tw.intro}, as well as the anonymous referees for their careful reading and comments, which helped improve the presentation.
A.P.'s research is funded by the European Research Council (ERC) through StG 101165368 ``MAGNETIC''. G.D.P.'s research is funded by the European Research Council (ERC) through CoG   101169953 ``RISE''.\footnote{Views and opinions expressed are however those of the authors only and do not necessarily reflect those of the European Union or the European Research Council.}

\section{Preliminaries and proof of Propositions~\ref{equiv},~\ref{sumset}, and~\cref{usc}}\label{sec.setup}
Given two integers \(n\ge3\) and \(k\in\{2,\dots,n-1\}\), let \(\gr\) be the Grassmannian of \(k\)-planes in \(\R^n\) (without orientation). Recall that it admits a natural structure of compact manifold of dimension \(k(n-k)\)
and can be identified with the set of matrices
\[\{S\in\R^{n\times n}\,:\,S^T=S,\ S^2=S,\ \operatorname{tr}S=k\},\]
with the smooth structure induced from \(\R^{n\times n}\). Given a plane \(P\in\gr\) and a linear isomorphism
\(L\in\text{GL}(n)\), we denote by \(L[P]=L(P)\in\gr\) the image of \(P\) through \(L\).

Consider a \emph{\(k\)-dimensional varifold} (or simply \emph{\(k\)-varifold}) \(V\) in an open set \(U\subseteq\R^n\), i.e., a nonnegative Radon measure on \(U\times\gr\). Letting \(\pi:\R^n\times\gr\to\R^n\) denote the projection on the first factor, we will often write
\[|V|:=\pi_*V,\]
called the \emph{weight} of \(V\) (a Radon measure on \(U\)).
A diffeomorphism \(\phi:U\to U\) induces a diffeomorphism \(\hat\phi\) of \(U\times\gr\), mapping \((x,P)\)
to \((\phi(x),d\phi(x)[P])\). The \emph{varifold pushforward} of \(V\) through \(\phi\) is
\[\phi_*V:=\hat\phi_*(J_\phi V),\quad J_\phi(x,P):=\frac{\mathcal{H}^k(d\phi(x)[P\cap B_1(0)])}{\mathcal{H}^k(P\cap B_1(0))}=\frac{\mathcal{H}^k(d\phi(x)[P\cap B_1(0)])}{\omega_k},\]
where the Jacobian factor \(J_\phi:U\times\gr\to(0,\infty)\) is the usual correction factor motivated by the area formula.

Let \(F:\gr\to(0,\infty)\) be a \(C^1\) function, which is often called an \emph{anisotropic integrand}.
It induces a functional on \(k\)-varifolds in \(U\) given by
\[\mathcal{F}(V):=\int_{U\times\gr}F(P)\,dV(x,P).\]
More generally, given \(B\subseteq U\) Borel, we set
\[\mathcal{F}(V,B):=\int_{B\times\gr}F(P)\,dV(x,P);\]
it is clear that for two constants \(0<c(n,k,F)\le C(n,k,F)\) we have
\[c|V|(B)\le\mathcal{F}(V,B)\le C|V|(B).\]
Given a vector field \(X\in C^1_c(U,\R^n)\), the \emph{first variation of \(V\) along \(X\)} (with respect to \(F\)) is defined as
\begin{align*}
	&\ang{\delta^F V,X}:=\frac{d}{dt}\Big|_{t=0}\mathcal F((\varphi_t^X)_*V,\operatorname{spt}(X)),
\end{align*}
where \((\varphi_t^X)_{t\in\R}\) is the flow of \(X\) at time \(t\). It can be shown (see, e.g.,~\cite[Lemma A.2]{DDG}) that, under the above identification, the previous derivative always exists and is given explicitly by
\begin{equation}\label{fv.bf}
    \ang{\delta^F V,X}=\int_{U\times\gr}\ang{B_F(S),dX(x)}\,dV(x,S),
\end{equation}
where we use the Hilbert--Schmidt inner product on matrices and \(B_F(S)\in\R^{n\times n}\)
is uniquely defined by
\begin{equation}\label{bf}
    \ang{B_F(S),L}:=F(S)\ang{S,L}+dF(S)[(I-S)LS+SL^T(I-S)]
\end{equation}
for all \(L\in\R^{n\times n}\) (note that \((I-S)LS+SL^T(I-S)\in T_S\gr\)).

As in the isotropic case (where \(F\equiv 1\)), we will say that \(V\) has \emph{locally bounded first variation}
if \(\ang{\delta^F V,X}\) can be locally represented by integration of \(X\) against a vector-valued measure.
In this case, \(|\delta^FV|\) denotes the associated total variation measure.

We will eventually focus on the codimension-one case \(k=n-1\) (in particular, when \(k=2\) and \(n=3\)).
In this case, we can identify \(\text{Gr}_{n-1}(\R^n)\cong\mathbb{RP}^{n-1}\),
i.e., a hyperplane \(P\) is identified with \(\pm\nu\), the unit vector perpendicular to \(P\).
We can then identify \(F\) with an even function
\[F:\mathbb{S}^{n-1}\to(0,\infty).\]
Taking the \(1\)-homogeneous extension (still denoted by \(F\)), namely \(F(\lambda\nu):=\lambda F(\nu)\)
for all \(\lambda\ge0\), we obtain a function \(F:\R^n\to[0,\infty)\).
In this case, the \emph{atomic condition (AC)} mentioned in the introduction is equivalent to
requiring that \(F\) is strictly convex (more precisely, along all lines which do not pass through \(0\)), as shown in~\cite[Theorem 1.3]{DDG}.
Also, we have the simpler formula
\begin{equation}\label{bf.hyper}
    B_F(\nu)=F(\nu)I-\nu\otimes dF(\nu),
\end{equation}
where \(dF(\nu)\in(\R^n)^*\) is the differential of the \(1\)-homogeneous extension at \(\nu\) (note that \(B_F(\nu)=B_F(-\nu)\), as expected).

We now turn to the proofs of the general facts stated in the introduction, starting with a well-known lemma.

\begin{lemma}\label{trunc}
    Given a \(k\)-varifold \(V\) in an open set \(U\) with locally bounded first variation,
    if \(\bar B_r(p)\subset U\) and the derivative \(\frac{d}{d\rho}|V|(B_\rho(p))|_{\rho=r}\) exists then \(V':=\uno_{B_r(p)}V\) has
    \[|\delta^F V'|(\R^n)\le |\delta^F V|(B_r(p))+C(n,k,F)\frac{d}{d\rho}|V|(B_\rho(p))\Big|_{\rho=r}.\]
\end{lemma}

\begin{proof}
    This is easily seen by taking a cut-off function \(\chi\) such that \(\chi=1\) on \(B_{r-h}(p)\), \(\chi=0\) outside \(B_{r}(p)\), and \(|d\chi|\le{2}/{h}\) (for a given \(h\in(0,r)\)), and noting that the varifold \(\chi V\) has
	\begin{align*}
		&\ang{\delta^F(\chi V),X}=\ang{\delta^F V,\chi X}+O(\|d\chi\|_{C^0}\|X\|_{C^0})\cdot |V|(B_{r}(p)\setminus B_{r-h}(p))
	\end{align*}
	for any \(X\in C^1_c(\R^n,\R^n)\), so that
	\begin{align*}
		&|\delta^F(\chi V)|(\R^n)\le |\delta^FV|(B_r(p))+C(n,k,F)\frac{|V|(B_{r}(p))-|V|(B_{r-h}(p))}{h},
	\end{align*}
    which gives the claim in the limit \(h\to 0\).
\end{proof}

\begin{proof}[Proof of~\cref{usc}]
    Let us prove (i) assuming $\theta_0=1$ (up to normalizing $V$). Letting \(\mu(r):=|V|(B_r(x))\), since for a.e.\ \(r>0\) the truncated varifold
\(V':=\uno_{B_r(x)}V\) has \(|\delta^FV'|(\R^n)\le|\delta^F V|(B_r(x))+C\mu'(r)\), we get
\[\mu(r)^{(k-1)/k}\le C|\delta^FV|(B_r(x))+C\mu'(r)\le C_x\mu(r)^\alpha+C(n,k,F)\mu'(r)\]
for a.e.\ \(r\in(0,1)\), where \(\alpha:=1-\frac1p\) and \(C_x\) depends on \(n,k,F\) and also on (an upper bound on) \(\|f\|_{L^p(B_r(x),|V|)}\), where \(|\delta^FV|=f|V|\).
Since \(x\in\operatorname{spt}(|V|)\), we have \(\mu(r)>0\) and hence
\[(\mu(r)^{1/k})'\ge \frac{c-C_x'\mu(r)^\beta}{k}\]
for a.e.\ \(r\in(0,1)\),
where \(c=C(n,k,F)^{-1}\) and \(\beta:=\alpha-(1-\frac1k)>0\). As long as \(C_x'r^{\beta k}<c/2\), we either have
\(\mu(r)\ge r^k\) or \(\frac{c-C_x'\mu(\rho)^\beta}{k}\ge\frac{c}{2k}\) for all \(\rho\in(0,r)\). In the latter case, we obtain \(\mu(r)^{1/k}\ge\frac{cr}{2k}\), giving the claim in both cases (for a different \(c>0\)).
The upper semicontinuity of the support along converging sequences is a direct consequence.
\end{proof}

\begin{proposition}\label{functional}
    The geometric version~\cref{ms} of the Michael--Simon bound is equivalent to the functional version~\cref{ms.f}.
\end{proposition}

\begin{proof}
    To see that~\cref{ms} implies~\cref{ms.f}, note that a simple cut-off argument as in the previous proof shows that, for a.e.\ \(t>0\),
\(V':=\uno_{\{|f|>t\}}V\) has
\[|\delta^FV'|(\R^n)\le |\delta^FV|(\{|f|>t\})-Ch'(t),\quad h(t):=\int_{\{|f|>t\}}|df|\,d|V|,\]
for a possibly different \(C\). Assuming for simplicity \(\theta_0=1\) and applying~\cref{ms} to \(V'\), we get
\[m(t)^{(k-1)/k}\le C|\delta^FV|(\{|f|>t\})-Ch'(t),\quad m(t):=|V|(\{|f|>t\}),\]
which gives~\cref{ms.f} thanks to the well-known bound
\[
 \Big(\int_0^\infty m(t)t^{p-1}\,dt\Big)^{1/p} \le p^{-1/p}\int_0^\infty m(t)^{1/p}\,dt
\]
for any decreasing \(m(t)\)
and \(p\in[1,\infty)\) (we take \(p:=k/(k-1)\)).

Conversely, given \(R>0\), let \(f_R\in C^1_c(\R^n)\) with \(0\le f_R\le1\), \(f_R=1\) on \(B_R(0)\),
and \(|df_R|\le 2/R\). Applying~\cref{ms.f} to \(f_R\) and letting \(R\to\infty\), since
\(\int_{\R^n}|df_R|\,d|V|\le 2|V|(\R^n)/R\to0\) we obtain
\(|V|(\R^n)^{(k-1)/k}\le C'\theta_0^{-1/k}|\delta^FV|(\R^n)\), which is~\cref{ms} for a possibly different constant \(C\).
\end{proof}

Let us now show~\cref{sumset}, from which~\cref{equiv} will follow quite easily.

\begin{proof}[Proof of~\cref{sumset}]
    A similar argument appears in the proof of~\cite[Theorem 4.1]{DeRosa}.
    Up to a subsequence, we can assume that \(|\delta^FV_i|\weakto\nu\) for a suitable Radon measure \(\nu\).
    For \(|V|\)-a.e.\ \(x\), the rectifiable varifold \(V\) admits an approximate tangent plane
    and there exists \(C>0\), depending on \(x\), such that\footnote{Recall that, for any two Radon measures \(\nu,\mu\ge0\), we have \[\limsup_{r \to 0} \frac{
    \nu(B_r(x))}{\mu(B_r(x))}<\infty
    \quad \text{for \(\mu\)-a.e.\ \(x\).}\]}
    \[\nu(B_r(x))\le C|V|(B_r(x))\quad\text{for all }r\in(0,1).\]
    Take any such point \(x_0\); up to a translation and a rotation, we can assume that
    \(x_0=0\) and the tangent plane is \(\theta\) times \(P=\operatorname{span}\{e_1,\dots,e_k\}\),
    for some constant \(\theta>0\).

    For all but countably many \(r\in(0,1)\) we have \(|V|(\de B_r(0))=\nu(\de B_r(0))=0\), so that
    \[
     |V_i|(B_r(0))\to|V|(B_r(0))>0,\quad |\delta^FV_i|(B_r(0))\to\nu(B_r(0))\le C|V|(B_r(0)).
    \]
    Hence, by a straightforward diagonal argument, up to a subsequence we can find radii
    \(r_i\to0\) such that, denoting by \(W_i\) the varifolds dilated by a factor \(r_i^{-1}\) and
    \(\tilde E_i:=r_i^{-1}E_i\), we have
    \[W_i\weakto\theta P,\quad |W_i|(\tilde E_i\cap B_1(0))\to0,\quad
    |\delta^FV_i|(B_{r_i}(0))\le C|V_i|(B_{r_i}(0)),\]
    for a possibly different \(C\), where \(P\) is identified with the corresponding multiplicity-one varifold.
    To conclude, it suffices to show that \(\theta=\lim_{i\to\infty}v_i\) for suitable \(v_i\in S(\mathcal{D}_i)\).
    Clearly, it is enough to check that this holds along a subsequence.

    Note that
    \begin{equation}\label{resc.f.v}
     |\delta^FW_i|(B_1(0))=r_i^{1-k}|\delta^FV_i|(B_{r_i}(0))\le Cr_i^{1-k}|V_i|(B_{r_i}(0))
     =Cr_i|W_i|(B_1(0))\to0.
    \end{equation}
    Moreover, for each \(i\in\N\) we can find \(\rho_i\in(1/2,3/4)\) such that
    \[
     \frac{d}{d\rho}|W_i|(B_\rho(0))\Big|_{\rho=\rho_i}\le C|W_i|(B_1(0))\le C,
    \]
    so that by~\cref{trunc} the varifold \(W_i':=\uno_{B_{\rho_i}(0)}W_i\) has \(|\delta^FW_i'|(\R^n)\le C\).
    We can now apply~\cite[Lemma 3.2]{DDG} to the sequence \((W_i')\) (with \(F_i:=F\)) and deduce the
    strong convergence
    \[(\pi_P)_*|W_i'|\to f\mathcal{H}^k\mrestr P,\]
    for some \(f\in L^1(P)\) supported in the unit ball. Moreover, since \(W_i\weakto\theta P\), we have
    \(f=\theta\) on \(P\cap B_{1/2}(0)\).

    Now, denoting by \(T_xW_i'\in\gr\) the (normalized) approximate tangent plane (which exists \(|W_i'|\)-a.e.),
    let \(J_i(x)\in[0,1]\) denote the Jacobian of the projection \(\pi_P\) along \(T_xW_i'\), namely
    \(J_i(x)=\omega_k^{-1}\mathcal{H}^k(\pi_P(B_1(0)\cap T_xW_i'))\).
    By the varifold convergence \(W_i\weakto\theta P\), we have
    \[\int_{B_1(0)}|J_i(x)-1|\,d|W_i'|\to0,\]
    so that
    \[(\pi_P)_*[J_i|W_i'|\mrestr(\R^n\setminus\tilde E_i)]\to f\mathcal{H}^k\mrestr P.\]
    Since \(\Theta^k(|W_i'|,x)\in\mathcal{D}_i\) for \(|W_i'|\)-a.e.\ \(x\in\R^n\setminus\tilde E_i\),
    by the area formula we have
    \[(\pi_P)_*[J_i|W_i'|\mrestr(\R^n\setminus\tilde E_i)]=f_i\mathcal{H}^k\mrestr P\]
    for some \(f_i\in L^1(P)\) taking values in \(S(\mathcal{D}_i)\cup\{0\}\) a.e.\ (as \(\inf\mathcal{D}_i>0\) for all \(i\)).
    Since \(f_i\to\theta\) strongly in \(L^1(B_{1/2}(0))\), the claim follows.
\end{proof}

Finally, let us turn to the equivalence between Michael--Simon and compactness of rectifiable varifolds,
namely~\cref{equiv}.

\begin{proof}[Proof of~\cref{equiv}]
    Let us first check that (i) implies (ii): let us then assume that \(V_i\weakto V\) is a sequence as in (ii), with \(\theta_0=1\). Given a point \(p\in\operatorname{spt}(|V_i|)\), if \(|\delta^FV_i|(B_r(p))\le\Lambda|V_i|(B_r(p))\) for all \(r\in(0,1/2)\) then we claim that~\cref{ms} gives \(c(n,k,F,\Lambda)>0\) such that
    \begin{equation}\label{linf.lb}
        |V_i|(B_r(p))\ge cr^k\quad\text{for all }r\in(0,1/2).
    \end{equation}
    Indeed, letting \(\mu(r):=|V_i|(B_r(p))\),
	whenever the classical derivative \(\mu'(r)\) exists,~\cref{trunc} says that the varifold \(V_i':=\uno_{B_r(p)}V_i\) has total first variation at most \(\Lambda\mu(r)+C(n,k,F)\mu'(r)\).
	Now, using assumption (i), namely~\cref{ms}, for \(V_i'\) (which can be viewed as a varifold in \(\R^n\) even when \(M=\mathbb{T}^n\), as \(B_{1/2}(p)\subset\mathbb{T}^n\) is isometric to \(B_{1/2}(0)\subset\R^n\)), we deduce that
	\begin{align*}
		&\mu(r)^{(k-1)/k}\le \Lambda \mu(r)+C\mu'(r)\quad\text{for a.e.\ }r\in(0,1/2),
	\end{align*}
    which easily gives~\cref{linf.lb} (see also the proof of~\cref{usc}).

	We now repeat the first part of the proof of~\cite[Theorem 40.6]{Simon}.  Let \(K:=\bar B_R(0)\) and call \(S_{i,\ell}\subseteq K\) the set of points \(p\in \operatorname{spt}(|V_i|)\cap K\) such that \(|\delta^F V_i|(B_r(p))\le\ell|V_i|(B_r(p))\) for all \(0<r<1/2\).
	Since \(\sup_i|\delta^F V_i|(B_{R+1})<\infty\), using Besicovitch's covering lemma it is immediate to check that \(|V_i|(K\setminus S_{i,\ell})\to 0\) as \(\ell\to\infty\), uniformly in \(i\).

	It follows that, letting \(S_\ell\) denote the (relatively closed) set of points \(q\) in \(B_R(0)\) such that \(q=\lim_{j\to\infty}p_{i_j}\), for some sequence \(p_{i_j}\in S_{i_j,\ell}\) (and some subsequence \(i_j\to\infty\)), we have
	\[\lim_{\ell\to\infty}|V|(B_R(0)\setminus S_\ell)=0.\]
	Indeed, for any compact \(K'\subseteq B_R(0)\setminus S_\ell\) there exists a small \(\rho>0\) such that
    \(S_{i,\ell}\cap B_\rho(K')=\emptyset\) for \(i\) large enough, yielding \(|V|(K')\le\liminf_{i\to\infty}|V_i|(B_\rho(K')\setminus S_{i,\ell})\) and thus \(|V|(B_R(0)\setminus S_\ell)\le\liminf_{i\to\infty}|V_i|(K\setminus S_{i,\ell})\to0\)
    as \(\ell\to\infty\).

    On the other hand, given \(q\in S_\ell\),~\cref{linf.lb} gives \(|V|(\bar B_r(q))\ge c(n,k,F,\ell)r^k\) for all \(r\in(0,1/2)\), and thus
    \[\Theta^k_*(|V|,q)>0\quad\text{for all }q\in\bigcup_\ell S_\ell.\]
    Since the complement of \(\bigcup_\ell S_\ell\) is \(|V|\)-negligible, by~\cite[Theorem 1.2]{DDG}
    the varifold \(V\) is rectifiable. Moreover, by~\cref{sumset} it has density \(\ge1\) at \(|V|\)-a.e.\ point, as desired.

    It is obvious that (ii) implies (iv). Moreover, we claim that (iv) implies (iii):
    if we have a bad sequence as in (iii), then by averaging we can find \(c_1^i,\dots,c_n^i\in\R/\Z\)
    such that
    \[
     |V_i|(\pi_j^{-1}(B_r(c_j^i)))\le 4r
    \]
    for any \(r\) small enough (depending on \(i,j\)),
    where \(\pi_j:\mathbb{T}^n\to\mathbb{T}^1\) is the projection to the \(j\)-th coordinate.
    Up to a translation, we can assume that \(c^i_j=0\). Then, lifting \(V_i\) to a periodic varifold \(\tilde V_i\)
    in \(\R^n\), the truncated varifolds \(\uno_{(0,1)^n}\tilde V_i\) have uniformly bounded first variation
    (by the same argument of~\cref{trunc}) and satisfy all the other conclusions of a bad sequence in (iv).

    Finally, let us see that (iii) implies (i). Assuming by contradiction that~\cref{ms} fails,
    we will construct a bad sequence as in (iii). Fix a finite \(\alpha>\frac{n+1}{k}\); since~\cref{ms} is scale-invariant, if it fails then
    there exists a sequence \((V_i)\) of rectifiable \(k\)-varifolds with \(\Theta^k(|V_i|,x)\ge1\) for \(|V_i|\)-a.e.\ \(x\) and
    \[|V_i|(\R^n)=1,\quad |\delta^FV_i|(\R^n)\le\epsilon_i^\alpha,\]
    for a vanishing sequence \(0<\epsilon_i\le1/i\).
    Dilating \(V_i\) by a factor \(\epsilon_i^{(n+1)/k}<1\) and multiplying the resulting varifold by \(\epsilon_i^{-1}\),
    we obtain a new sequence of varifolds \(W_i\) satisfying
    \[|W_i|(\R^n)=\epsilon_i^n,\quad|\delta^FW_i|(\R^n)\le\epsilon_i^\beta,\quad\Theta^k(|W_i|,x)\ge\epsilon_i^{-1}\ge i\text{ for \(|W_i|\)-a.e.\ \(x\)},\]
    where \(\beta:=\alpha+\frac{(n+1)(k-1)}{k}-1>n\).
    Finally, assuming without loss of generality that \(\epsilon_i=\frac{1}{\ell_i}\) for some integer \(\ell_i>1\), we define
    \[\tilde Z_i:=\sum_{c\in\Z^n}(W_i+\epsilon_ic)\]
    where, with a slight abuse of notation, we denote by \(V_i+a\) the translation of \(V_i\) by \(a\in\R^n\).
    It is easy to check that finite partial sums have locally uniformly bounded mass, so that the series does indeed define a \(k\)-varifold in \(\R^n\).
    This varifold is \(\Z^n\)-periodic and thus is the lift of a varifold \(Z_i\) on \(\mathbb{T}^n\), which has
    \[|Z_i|(\mathbb{T}^n)=\epsilon_i^{-n}|W_i|(\R^n),\quad|\delta^F Z_i|(\mathbb{T}^n)\le\epsilon_i^{-n}|\delta^F W_i|(\R^n)\]
    and hence
    \[|Z_i|(\mathbb{T}^n)=1,\quad|\delta^F Z_i|(\mathbb{T}^n)\le\epsilon_i^{\gamma},\quad\gamma:=\beta-n>0,\]
    as well as \(\Theta^k(|Z_i|,x)\ge i\) at \(|Z_i|\)-a.e.\ \(x\).
    Since in fact \(Z_i\) is \(\epsilon_i\Z^n\)-periodic, it is clear that it converges to a varifold of the form
    \(Z=\mathcal{L}^n\otimes\lambda\), with \(\lambda\in\mathcal{P}(\gr)\), along a subsequence.
\end{proof}

\section{Nonlinear inequalities for vector fields on the plane}\label{sec.nonlinear}

The main goal of this section is to prove the following new nonlinear inequality for vector fields on the plane, stated again for the reader's convenience.

\begin{theorem}\label{main.tw}
	Given two vector fields \(S,T\in W^{1,1}(\R^2,\R^2)\), we define
	\begin{align}\label{sp.sn}
		\begin{aligned}
		S^P&:=(S^x-|S^y|)^+,\quad S^N:=(S^x-|S^y|)^-, \\
		T^P&:=(T^y-|T^x|)^+,\quad T^N:=(T^y-|T^x|)^-.
		\end{aligned}
	\end{align}
    Also, let \(\chi:\R^2\to[0,1]\) be a Borel function supported in a bounded set.
    Then, for some universal constant \(C>0\), we have
	\begin{align}\label{kak.bis}
        \begin{aligned}
		\int_{\R^2} \chi \min\{S^P,T^P\}
		&\le C\|\chi\|_{L^2}\Big[\int_{\R^2}(|S|+\lvert \dive S\rvert)\Big]^{1/2}\Big[\int_{\R^2}(|T|+\lvert \dive T\rvert)\Big]^{1/2} \\
		&\quad+C\int_{\R^2} (\lvert \dive S\rvert+\lvert \dive T\rvert)
        +\int_{\R^2}(S^N+T^N).
        \end{aligned}
	\end{align}
    The same holds if \(S^x,S^y,T^x,T^y,\dive S,\dive T\) are just real-valued measures on the plane with finite total variation.
\end{theorem}

\begin{remark}
    In the general case of measures, quantities such as \(\int_{\R^2}|S|\) should be interpreted as \(|S|(\R^2)\) and \(S^P,S^N,T^P,T^N\) are defined by the same formulas~\cref{sp.sn}
    (equivalently, writing \(S=\sigma|S|\) for a unit-valued \(\sigma\), we have \(S^P=(\sigma^x-|\sigma^y|)^+|S|\)). Recall also that \(\min\{\mu,\nu\}:=(\mu+\nu-|\mu-\nu|)/2\)
    for two real-valued measures \(\mu,\nu\) (equivalently, writing \(\mu=f(|\mu|+|\nu|)\) and
    \(\nu=g(|\mu|+|\nu|)\), we have \(\min\{\mu,\nu\}=\min\{f,g\}(|\mu|+|\nu|)\)).
\end{remark}

While~\cref{kak.bis} is not scale-invariant, taking an arbitrary \(\chi\) such that \(\chi=0\) \(\mathcal{L}^2\)-a.e.\ and applying the bound to all rescalings of \(S,T\), we immediately obtain the following,
which can be seen as a special case of~\cite[Proposition 8.6]{ACP}.

\begin{corollary}
    Under the same assumptions, denoting by \(\mathfrak{S},\mathfrak{T}\) the singular parts of the measures \(S,T\), we have
    \[\int_{\R^2} \min\{\mathfrak{S}^P,\mathfrak{T}^P\}\le \int_{\R^2}(S^N+T^N).\]
    In particular, if \(S^N=T^N=0\) then the two measures \(\mathfrak{S}^P\) and \(\mathfrak{T}^P\)
    are mutually singular.
\end{corollary}

Thus, in the last corollary, the assumption that \(\dive S\) and \(\dive T\) are finite measures is used only qualitatively.
Given \(u\in BV(\R^2,\R^2)\) and taking
\[S:=(\de_yu^y,-\de_x u^y),\quad T:=(-\de_yu^x,\de_x u^x),\]
we have \(\dive S=\dive T=0\). Taking \(\varphi_r\) supported in a ball \(B_r(p)\) and applying the previous bound to \(\varphi_r S,\varphi_r T\), we obtain
\begin{align*}
&\int_{\R^2} \varphi_r \min\{(A_y^y-|A_x^y|)^+,(A_x^x-|A_y^x|)^+\}\,d|Du|^s\\
&\le \int_{\R^2}\varphi_r[(A_y^y-|A_x^y|)^- + (A_x^x-|A_y^x|)^-]\,d|Du|,
\end{align*}
where \(A\) (with rows \(A^x,A^y\)) is given by the polar decomposition \(Du=A|Du|\) and \(|Du|^s\) denotes the singular part. In particular, if \(p\) is an approximate continuity point for \(A\) and a point of density one for \(|Du|^s\) (with respect to \(|Du|\)), taking \(r\to0\) we see that \(A(p)\) must be far from \(I\). Since this must also hold if we compose \(u\) with linear transformations, this immediately implies the following.

\begin{corollary}[Alberti's rank-one theorem~\cite{Alberti,DR,MV}]
    Given \(u\in BV(\R^2,\R^2)\) and writing \((Du)^s=A|Du|^s\), we have \(\operatorname{rk}(A)=1\)
    at \(|Du|^s\)-a.e.\ point.
    By a well-known slicing argument (see, e.g.,~\cite[Proposition 1.3]{DeLellis}), this implies that the same holds for a function \(u\in BV(\R^n,\R^m)\) with any \(m,n\ge1\).
\end{corollary}

Before turning to~\cref{main.tw}, we will first obtain a simpler and perhaps more intuitive bound,
for vector fields obeying a certain nonlinear constraint.
Namely, we will derive the following sharp estimate, which was also proved in~\cite[Theorem A]{GRS} (although with non-sharp constant).
We will give a full proof of it since our techniques are different and more readily adaptable to
give also~\cref{main.tw} above.

\begin{theorem}\label{det.simple}
	Given \(S,T\in W^{1,2}(\R^2,\R^2)\), assume that \(S^x,T^y\ge 0\) and \(\det(S,T)\ge 0\) a.e. Then
	\begin{align*}
		&\int_{\R^2}\det(S,T)\le\frac{1}{4}\Big(\int_{\R^2}\lvert \dive S\rvert \Big)\Big(\int_{\R^2}\lvert \dive T\rvert\Big).
	\end{align*}
\end{theorem}

The intuition behind this result is that, by Smirnov's decomposition theorem for 1-dimensional currents~\cite{Smirnov}, the current associated to \(S\) is a (weighted) superposition of curves \(\gamma\). The boundary of each curve contributes \(2\) to the total mass of the boundary, which is \(\int_{\R^2}\lvert \dive S\rvert\), so that (informally) the weighted number of curves building up \(S\) is \(\mz\int_{\R^2}\lvert \dive S\rvert\), and similarly for \(T\). On the other hand, the condition \(\det(S,T)\ge 0\) forces a curve in \(S\) and a curve in \(T\) to meet at most once, and the integral of \(\det(S,T)\) counts the (weighted) total number of intersections. This is best seen by looking at an example where \(S,T\in BV(\R^2,\R^2)\) are supported on \(\epsilon\)-fattened curves, for \(\epsilon\) small, and point along the curve; this example also shows that the bound is sharp, with equality achieved quite often. With this intuition in mind,
this bound can be seen as a functional version of the (trivial) planar case of the multilinear Kakeya inequality~\cite{BCT,Guth}, extended to a situation where the tubes are not necessarily straight.

\vspace{5pt}
	\begin{center}
		\begin{overpic}[width=10cm]{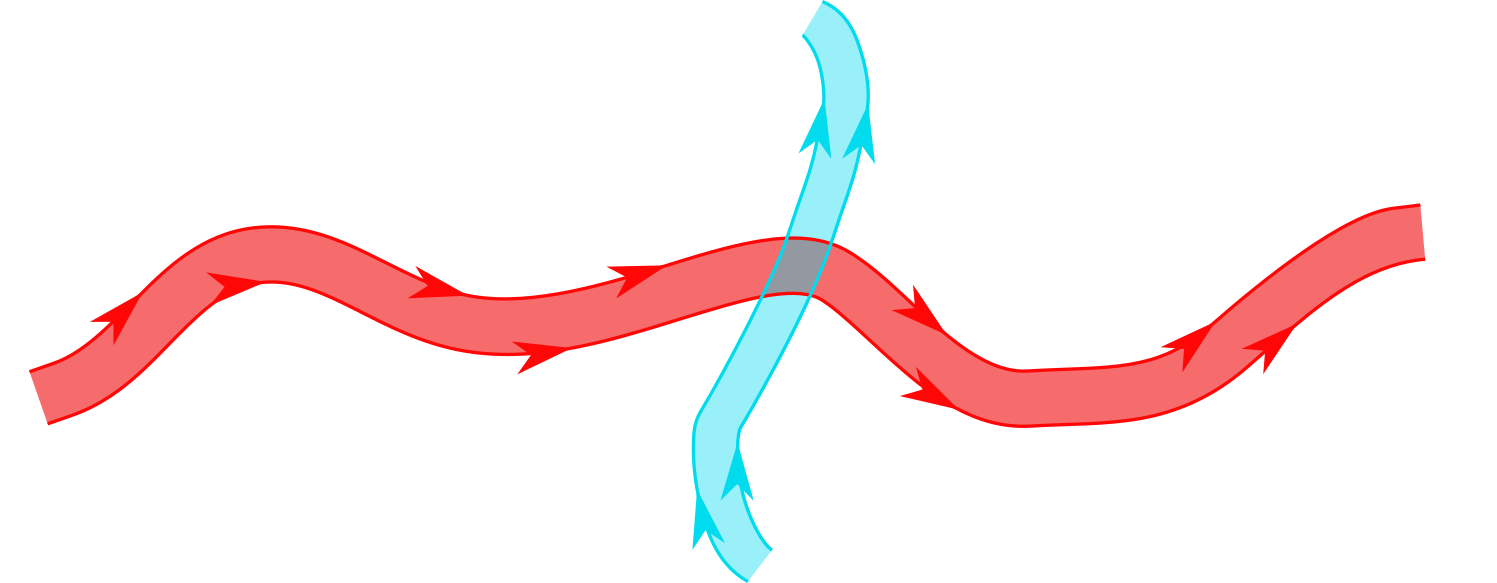}
			\put(20,25){{\color{red}\(S\)}}
			\put(60,30){{\color{cyan}\(T\)}}
		\end{overpic}
	\end{center}
\vspace{5pt}

\subsection{Proof of the simpler~\cref{det.simple}}
\emph{Step 1.} In order to prove~\cref{det.simple}, we first notice that we can reduce to the case of smooth, compactly supported vector fields. Indeed, we can assume that both \(S,T\)
are supported in a ball \(B_R(0)\) and \(|S|^2+|T|^2\le\Lambda^2\) for some \(R,\Lambda>0\). Let us fix a nonnegative cut-off function \(\psi\in C^\infty_c(\R^2)\)
such that \(\psi=1\) on \(B_{R+1}(0)\). Replacing \(S^x\) and \(T^y\) with \(S^x+\epsilon|S^y|+\epsilon\psi\)
and \(T^y+\epsilon|T^x|+\epsilon\psi\) (for \(\epsilon>0\) small), respectively, we obtain perturbed vector fields (still denoted by \(S,T\)) such that
\[S^x\ge\epsilon|S|,\quad T^y\ge\epsilon|T|,\]
the angle between \(S\) and \(T\) is \(\ge\epsilon'>0\)
(at points where \(S,T\neq0\)), and
\[S^x,T^y\ge\epsilon\quad\text{on \(B_{R+1}(0)\)},\]
while \(S,T\) are smooth on the complement of \(B_R(0)\). In particular, we have \(\det(S,T)\ge\epsilon''>0\)
on \(B_{R+1}(0)\).
Calling \(A_x,A_y\) the two rows of a generic \(2\times 2\) matrix \(A\),
we select \(\lambda,\mu\in(0,\epsilon)\) and \(\nu\in(0,\epsilon'')\), and let
\[\mathcal{M}:=\{A\,:\,A_x^x\ge\lambda,\ A_y^y\ge\mu,\ \det(A)\ge\nu\}.\]

Next, we can apply the standard Schoen--Uhlenbeck trick: approximating the pair \((S,T)\) by a mollification with variable radius (leaving \((S,T)\) unchanged on the complement of \(B_{R+1}(0)\)), by the embedding \(W^{1,2}\hookrightarrow VMO\) the resulting pair \((S',T')\) still takes values in $\mathcal{M}$ (on \(B_{R+1}(0)\)). We can then take
$(S',T')$ as a smooth perturbation of \((S,T)\).

\emph{Step 2.} Assuming henceforth that \(S,T\in C^\infty_c(\R^2,\R^2)\), we claim that, up to another small perturbation, we can assume \(S=\alpha Z\) and \(T=\beta W\), for smooth, complete vector fields \(Z,W\) and coefficients \(\alpha,\beta\in C^\infty_c(\R^2)\) such that
\begin{align}\label{positive}
	&\alpha\ge0,\quad\beta\ge0,\quad\dive Z=\dive W=0,\quad Z^x>0,\quad W^y>0,\quad\det(Z,W)>0.
\end{align}
Indeed, given \(\epsilon>0\) and a nonnegative cut-off function \(\phi\in C^\infty_c(\R^2)\) with \(\phi=1\) near the supports of \(S\) and \(T\), we can let \(S':=\phi\tilde S\) and \(T':=\phi\tilde T\), with \(\tilde S:=S+\epsilon\de_x\) and \(\tilde T:=T+\epsilon\de_y\).
Since \(S',T'\) are arbitrarily close to \(S,T\) in the smooth topology, it is enough to prove the inequality for these new vector fields.

Also, we can write \(\tilde S=\tilde\alpha Z\) for a vector field \(Z\) with \(\dive Z=0\) and \(\tilde\alpha\) a smooth positive function. Indeed, the plane is foliated by the integral curves of \(\tilde S\) starting on the vertical axis \(\{0\}\times\R\), since \(\tilde S^x>0\) and, outside of a compact set, \(\tilde S=\epsilon\de_x\). Hence, we can let \(\tilde\alpha:=1\) on the vertical axis \(\{0\}\times\R\) and solve the equation \(\dive(\tilde\alpha^{-1}\tilde S)=0\) along all the integral curves \(\gamma:\R\to\R^2\), where it becomes the ordinary differential equation
\begin{align*}
	&\frac{d}{dt}(\tilde\alpha\circ\gamma)(t)=(\tilde\alpha\dive\tilde S)\circ\gamma(t).
\end{align*}
Now, letting \(\alpha:=\phi\tilde\alpha\) and \(\beta:=\phi\tilde\beta\), we arrive at \(S'=\alpha Z\) and \(T'=\beta W\), as desired. Note that~\cref{positive} also holds.
In the sequel, we replace \(S\) and \(T\) with the approximations \(S'\) and \(T'\).

\emph{Step 3.} It is convenient to further approximate \(S\) and \(T\) with piecewise divergence-free vector fields, as follows.
For \(\tau>0\), let \(\gamma_j:\R\to\R^2\) be the integral curve of \(Z\) with \(\gamma_j(0)=(0,j\tau)\), for \(j\in\Z\). Together with the vertical lines \(\ell_k:=\{k\tau\}\times\R\), these curves split the plane into a family of (open) regions \(\mathcal{P}_\tau:=(\reg_{jk})_{j,k\in\Z}\) diffeomorphic to the unit square, where \(\reg_{jk}\) is bounded by \(\gamma_j\), \(\gamma_{j+1}\), \(\ell_k\), and \(\ell_{k+1}\).

	\vspace{7pt}
	\begin{center}
		\begin{overpic}[width=8cm]{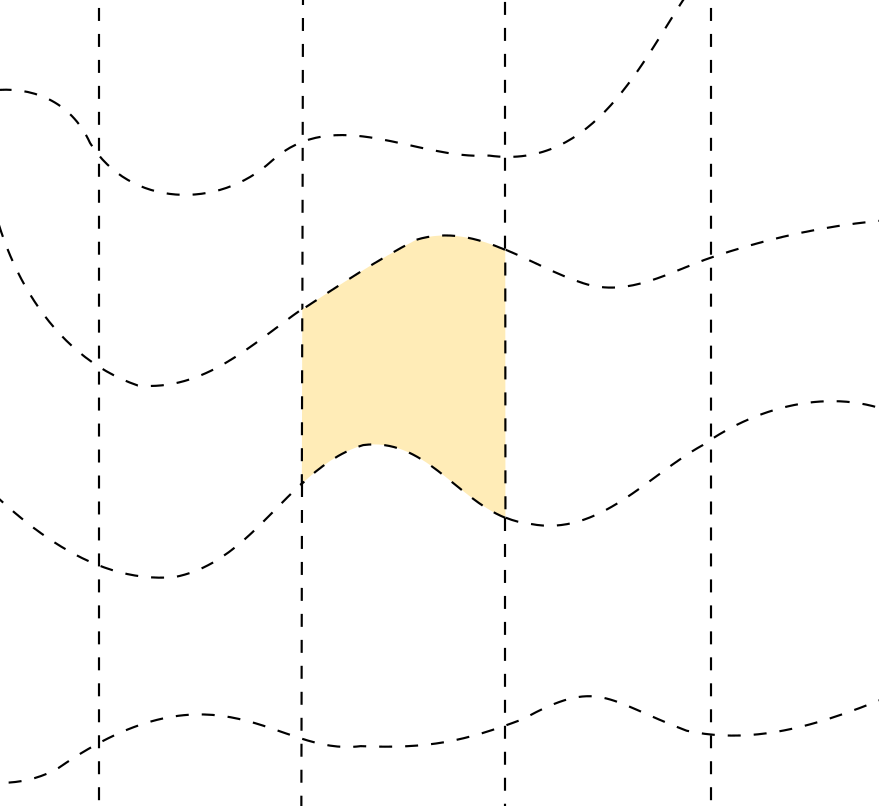}
			\put(36,87){\(\ell_k\)}
            \put(59,87){\(\ell_{k+1}\)}
            \put(0,38){\(\gamma_{j}\)}
			\put(0,67){\(\gamma_{j+1}\)}
			\put(42,49){{\color{orange}\(\reg_{jk}\)}}
		\end{overpic}
	\end{center}
	\vspace{7pt}

For each region \(\reg=\reg_{jk}\in\mathcal{P}_\tau\), we denote \(\de_L\reg:=\bar\reg\cap\ell_k\) and \(\de_R\reg:=\bar\reg\cap\ell_{k+1}\) the left and right sides of the boundary of \(\reg\).
We then let
\begin{align*}
	&S_\tau:=\sum_{\reg\in\mathcal{P}_\tau}\alpha_{\reg}\uno_\reg Z,
\end{align*}
with \(\alpha_{\reg}\ge 0\) a constant chosen so that the flow of \(\alpha_{\reg} Z|_\reg\) across \(\de_L\reg\) equals the flow of \(S|_\reg\) across the same segment.
Note that \(S_\tau\) belongs to \(BV(\R^2,\R^2)\) and that its divergence is a measure supported on the vertical lines \(\bigcup_k\ell_k\), since \(Z\) is divergence-free and parallel to the curves \(\gamma_j\).

\emph{Step 4.} Let us show a simple fact.

\begin{lemma}
	The total variation of \(\dive S_\tau\) on \(\R^2\) is bounded by \(\int_{\R^2}\lvert \dive S\rvert\).
\end{lemma}

\begin{proof}
	Given two adjacent regions \(\reg\) and \(\reg'\) with \(\de_R\reg=\de_L\reg'\), note that \(\alpha_{\reg'}\) is chosen so that
	\begin{align*}
		&\int_{\de_L\reg'}\alpha_{\reg'}Z^x=\int_{\de_L\reg'}S^x=\int_{\de_R\reg}S^x.
	\end{align*}
	Hence,
	\begin{align*}
		(\alpha_{\reg'}-\alpha_\reg)\int_{\de_L\reg'}Z^x
		&=\int_{\de_R\reg}S^x-\alpha_\reg\int_{\de_R\reg}Z^x \\
		&=\int_{\de_R\reg}S^x-\alpha_{\reg}\int_{\de_L\reg}Z^x \\
		&=\int_{\de_R\reg}S^x-\int_{\de_L\reg}S^x,
	\end{align*}
	where we used \(\dive Z=0\) in the second equality. Recalling that \(Z^x>0\), we get
	\begin{align*}
		&|\alpha_{\reg'}-\alpha_{\reg}|\int_{\de_L\reg'}Z^x
		=\Big|\int_{\de_R\reg}S^x-\int_{\de_L\reg}S^x\Big|
		\le\int_\reg\lvert \dive S\rvert.
	\end{align*}
	Since the total variation of \(\dive S_\tau\) is the sum of the quantity in the left-hand side over all pairs of regions \(\reg,\reg'\) with \(\de_R\reg=\de_L\reg'\), the claim follows.
\end{proof}

By the previous lemma, since \(S_\tau\to S\) pointwise as \(\tau\to 0\) (a.e.\ along a sequence), it is enough to prove~\cref{det.simple} for \(S_\tau\) and \(T_\tau\), where \(T_\tau\) is obtained in a similar way (interchanging \(x\) and \(y\)), relative to a partition \(\mathcal{Q}_\tau\) of the plane.

\emph{Step 5.} To motivate what follows, we make the following formal observation: since the coefficients in the approximation \(S_\tau=\sum_\reg\alpha_\reg\uno_\reg Z\) are nonnegative, \(S_\tau\) can be viewed as a superposition (with nonnegative coefficients) of indicator functions times \(Z\), in such a way that the divergence of these terms sums up to \(\dive S_\tau\) \emph{without cancellation}.

We now proceed to make this rigorous. Note that \(\alpha_\reg=0\) for all but finitely many regions.
Given \(\mu>0\) different from all the (finitely many) values \(\alpha_\reg\), for each \(j\in\Z\) consider a maximal chain of consecutive regions \(\reg_{jk},\dots,\reg_{jk'}\) such that the associated constants satisfy \(\alpha_\reg>\mu\), and let \(\hat\reg\) be their union, which coincides with the region bounded by \(\gamma_j\), \(\gamma_{j+1}\), \(\ell_k\), and \(\ell_{k'+1}\) (up to negligible sets).

It is easy to check that the sum \(s(\mu)\) of the total variations \(\lvert \dive(\uno_{\hat\reg}Z)\rvert(\R^2)\), as \(\hat\reg\) (and \(j\)) vary, equals the sum of \(\int_{\de_L\reg'}Z^x\) for all pairs of adjacent regions \(\reg,\reg'\) such that \(\mu\) lies between \(\alpha_{\reg}\) and \(\alpha_{\reg'}\) (regardless of the order). Hence,
\begin{align*}
	&\int_0^\infty s(\mu)\,d\mu
	=\sum_{\reg,\reg'\,:\,\de_R\reg=\de_L\reg}|\alpha_\reg-\alpha_{\reg'}|\int_{\de_L\reg'}Z^x
	=\lvert \dive S_\tau\rvert(\R^2).
\end{align*}
The analogous sums \(t(\nu)\) satisfy the analogous identity for the vector field \(W\), for a given value \(\nu>0\).

\emph{Step 6.} If we prove that
\begin{align}\label{det.simple.reg}
	&\int_{\hat\reg\cap\hat\reh}\det(Z,W)
	\le\frac{1}{4}\lvert \dive(\uno_{\hat\reg}Z)\rvert(\R^2)\,\lvert \dive(\uno_{\hat\reh}W)\rvert(\R^2)
\end{align}
for two chains of regions \(\hat\reg\) and \(\hat\reh\) as above (relative to \(Z\) and \(W\), respectively), then summing over all chains \(\hat\reg\) and \(\hat\reh\) (for a given choice of \(\mu\) and \(\nu\)) we get
\begin{align*}
	&\int_{\R^2}\det\Big(\sum_{\alpha_\reg>\mu}\uno_\reg Z,\sum_{\beta_\reh>\nu}\uno_\reh W\Big)
	\le\frac{1}{4}s(\mu)t(\nu),
\end{align*}
and~\cref{det.simple} follows by integrating in \(\mu\) and \(\nu\).

\emph{Step 7.} In order to prove~\cref{det.simple.reg}, we first find two smooth functions \(f,g:\R^2\to\R\) such that
\begin{align*}
	&Z=\nabla^\perp f,\quad W=\nabla^\perp g,
\end{align*}
where \(\nabla^\perp=(-\de_y,\de_x)\). The function \(f\) is just any primitive of the closed form \(Z^y\,dx-Z^x\,dy\), and similarly for \(g\). Note that the level sets of \(f\) and \(g\) are precisely the (maximal) integral curves of \(Z\) and \(W\).

The key observation is that, since \(\det(Z,W)>0\), an integral curve of \(Z\) meets an integral curve of \(W\) only once
(since \(\det(Z,W)=-df(W)\), actually \(f\) decreases along an integral curve of \(W\)).
Hence, the map \((f,g):\R^2\to\R^2\) is injective. Since \(\det(Z,W)\ge 0\) is the Jacobian determinant of this map,
by the area formula the integral of \(\det(Z,W)\) is bounded by the area of the image:
\begin{align*}
	&\int_{\hat\reg\cap\hat\reh}\det(Z,W)
	\le\mathcal{L}^2((f,g)(\hat\reg\cap\hat\reh)).
\end{align*}

Since \(\hat\reg\) is foliated by level sets of \(f\) starting on the left side \(\de_L\hat\reg=\de_L\reg_{jk}\), the oscillation of \(f\) on \(\hat\reg\), namely \(\sup_{\hat\reg}f-\inf_{\hat\reg}f\), is bounded by (and in fact equal to)
\begin{align*}
	&\int_{\de_L\hat\reg}|\de_y f|
	=\int_{\de_L\hat\reg}Z^x
	=\mz\lvert \dive(\uno_{\hat\reg}Z)\rvert(\R^2),
\end{align*}
where we used the fact that \(\int_{\de_L\hat\reg}Z^x=\int_{\de_R\hat\reg}Z^x\) in the last equality.
Hence, the image of \((f,g)|_{\hat\reg\cap\hat\reh}\) is included in a rectangle of area
\begin{align*}
	&\frac{1}{4}\lvert \dive(\uno_{\hat\reg}Z)\rvert(\R^2)\,\lvert \dive(\uno_{\hat\reh}W)\rvert(\R^2),
\end{align*}
which proves~\cref{det.simple.reg} and thus~\cref{det.simple}.

\subsection{A more general bound}
We present here an intermediate version, which will be instrumental in obtaining~\cref{main.tw}.
Let us start by introducing some additional terminology.
Given \(j,k\in\Z\), we let
\[\str_j^x:=[j,j+1]\times\R,\quad \str_k^y:=\R\times[k,k+1],\]
called \emph{vertical} and \emph{horizontal} stripes, respectively.
Assume now that \(S,T:\R^2\to\R^2\) are smooth vector fields.

\begin{definition}\label{good.region}
We say that \(p\in\{j\}\times\R\) is a \emph{good initial condition} for \(S\) if, for
the integral curve \(\gamma_p\) of \(S\) with \(\gamma_p(0)=p\),
there exists \(\tau_p>0\) such that
\[\gamma_p(\tau_p)\in\{j+1\}\times\R,\quad\gamma_p(t)\in(j,j+1)\times\R\text{ for all }t\in(0,\tau_p),\]
as well as \(\dot\gamma_p^x(0),\dot\gamma_p^x(\tau_p)>0\).
We say that \(\mathcal{R}^x\subset\str_j^x\) is a \emph{good region} for \(S\)
if there exists a segment \(\mathcal{I}:=\{j\}\times[a,b]\) (with \(a<b\)) on the left side,
consisting of good initial conditions for \(S\), such that
\[\mathcal{R}^x=\bigcup_{p\in\mathcal{I}}\gamma_p([0,\tau_p]).\]
The same notions for \(T\) are defined analogously, interchanging the roles of \(x,y\).
\end{definition}

Note that, since each \(\gamma_p\) meets \(\de\str_j^x\) transversely, \(\tau_p\)
depends smoothly on \(p\). In particular, a good region \(\mathcal{R}^x\)
is diffeomorphic to a compact rectangle \([a,b]\times[0,1]\) and \(S(q)\neq0\)
for all \(q\in\mathcal{R}^x\).

Given two good regions \(\mathcal{R}^x,\mathcal{R}^y\) for \(S,T\), respectively,
as in the proof of~\cref{det.simple} we can construct smooth vector fields
\(Z:\mathcal{R}^x\to\R^2\) and \(W:\mathcal{R}^y\to\R^2\), as well as smooth functions \(\alpha:\mathcal{R}^x\to\R\)
and \(\beta:\mathcal{R}^y\to\R\), such that
\[\dive Z=0,\quad\dive W=0,\quad S=\alpha Z,\quad T=\beta W,\quad \alpha>0,\quad \beta>0.\]
Note that we are \emph{no longer} assuming that \(S^x\ge0\) and \(T^y\ge0\), nor that \(\det(S,T)\ge0\)
on \(\mathcal{R}^x\cap\mathcal{R}^y\) (although we do have \(S^x>0\) on the left and right sides of \(\mathcal{R}^x\),
and similarly for \(W\)).

\begin{definition}\label{g.def}
	We define the cones
	\begin{align*}
		&C^x:=\{(x,y):|y|<x\},\quad C^y:=\{(x,y):|x|<y\}.
	\end{align*}
	Also, given \(p\in\mathcal{R}^x\), consider the (maximal) integral curve \(\gamma_p:I\to\R^2\) of \(S\) with initial condition \(\gamma_p(0)=p\) and let \(\tau_p\ge 0\) be the first time such that \(\gamma_p^x(\tau_p)=j+1\).
	We then define \(G^x\subseteq\mathcal{R}^x\) to be the Borel set of points \(p\) such that \(\gamma_p(t)\in p+C^x\) for all \(t\in(0,\tau_p]\). The set \(G^y\subseteq\mathcal{R}^y\) is defined analogously (with \(T\) and \(C^y\) in place of \(S\) and \(C^x\)). Finally, we let
	\begin{align*}
		&G:=G^x\cap G^y.
	\end{align*}
\end{definition}

The following picture illustrates a typical point \(p\in G^x\).

	\vspace{5pt}
	\begin{center}
		\begin{overpic}[width=8cm]{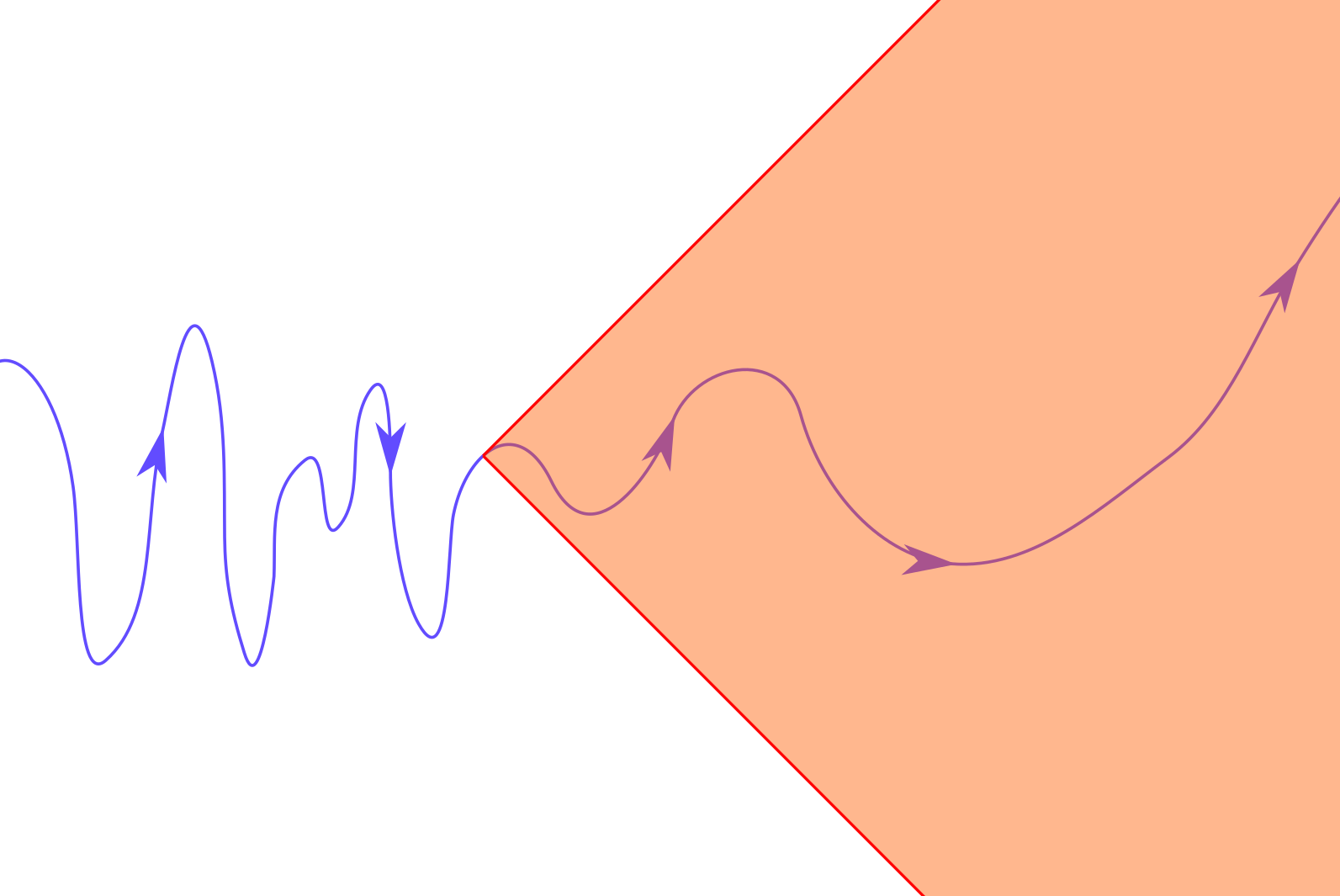}
			\put(2,45){{\color{blue}flow line of \(S\)}}
			\put(33,34){\(p\)}
			\put(65,54){{\color{red}cone \(p+C^x\)}}
		\end{overpic}
	\end{center}
	\vspace{5pt}

\begin{remark}\label{key}
	The key property of this set is that, whenever \(p,q\in G\), the two points cannot belong to the same integral curves for \(S\) and \(T\), unless \(p=q\). Indeed, if they were distinct points on the same integral curve of \(S\), then either \(q\in p+C^x\) or \(p\in q+C^x\); in both cases, it would follow that \(|p^y-q^y|<|p^x-q^x|\). The same argument for \(T\) would give the reverse inequality and thus a contradiction.
\end{remark}

We will prove that, for modified vector fields \(S_f\) and \(T_f\), the inequality
\begin{align}\label{g.est}
	&\int_{G}\det(S_f,T_f)\le\int_{\mathcal{R}^x}(S^x+\lvert \dive S\rvert)\int_{\mathcal{R}^y}(T^y+\lvert \dive T\rvert)
\end{align}
holds.
In order to reach this inequality, we need to localize the proof of~\cref{det.simple}, as follows.

Given any \(p\in\mathcal{R}^x\), we can write \(p=\gamma(t)\) for a unique (piece of) integral curve \(\gamma:[0,T]\to\mathcal{R}^x\) of \(Z\) starting on the left side \(\de_L\mathcal{R}^x\) of \(\mathcal{R}^x\) and ending on the right side (namely, \(\gamma^x(0)=j\) and \(\gamma^x(T)=j+1\)). We then let
\begin{align}\label{s.f.def}
	&\alpha_f(p):=\min_{[0,T]}(\alpha\circ\gamma)>0,\quad S_f:=\alpha_f Z.
\end{align}
Similarly we define \(\beta_f\) and \(T_f\) on \(\mathcal{R}^y\).

Since \(\alpha_f\) is constant along integral curves of \(Z\), the (continuous) vector field \(\alpha_f Z\) is still divergence-free on \(\mathcal{R}^x\).
As in the previous subsection, we claim that
\begin{align}\label{g.simple.est}
	&\int_{G}\det(S_f,T_f)\le\Big(\int_{\de_L\mathcal{R}^x}S^x\,d\hau^1\Big)\Big(\int_{\de_B\mathcal{R}^y}T^y\,d\hau^1\Big),
\end{align}
where the last two integrals are on the left and bottom sides, respectively.
Indeed, since \(S_f\) is divergence-free and \(\mathcal{R}^x\) is simply connected, we can write \(S_f=\nabla^\perp\phi\), and similarly \(T_f=\nabla^\perp\psi\), for two \(C^1\) functions \(\phi\) and \(\psi\) (on \(\mathcal{R}^x\) and \(\mathcal{R}^y\), respectively).

Since \(S_f^x>0\) on the left side \(\de_L\mathcal{R}^x\),
\(\phi\) is injective here, and hence
each integral curve of \(S\) (in \(\mathcal{R}^x\)) is an entire level set of \(\phi\);
the same holds for \(T\) and \(\psi\).
By~\cref{key}, we can then assert that the map \((\phi,\psi)\) is injective on \(G\). Hence, by the area formula, \(\int_{G}\det(S_f,T_f)=\int_{G}\det(\alpha_f Z,\beta_f W)\) is bounded by the area of the image of this map.

The oscillation of \(\phi\) (namely, \(\sup_{\mathcal{R}^x}\phi-\inf_{\mathcal{R}^x}\phi\)) is bounded by (and actually equal to)
\begin{align*}
	&\int_{\de_L\mathcal{R}^x}|\de_y\phi|
	=\int_{\de_L\mathcal{R}^x}\alpha_f Z^x
	\le\int_{\de_L\mathcal{R}^x}\alpha Z^x
	=\int_{\de_L\mathcal{R}^x}S^x,
\end{align*}
where we used the fact that \(\alpha_f\le\alpha\). This, together with the same bound for the oscillation of \(\psi\), gives our claim~\cref{g.simple.est}.

Further, since \(S\) is tangent to the top and bottom sides of \(\mathcal{R}^x\)
(namely, to the boundary \(\de\mathcal{R}^x\) in the open stripe \((j,j+1)\times\R\)), we have
\[\int_{(\{j\}\times\R)\cap\mathcal{R}^x}S^x
=\int_{(\{s\}\times\R)\cap\mathcal{R}^x}S^x
-\int_{([j,s]\times\R)\cap\mathcal{R}^x}\operatorname{div}(S)\]
for a.e.\ \(s\in[j,j+1]\).
Averaging over this interval, we obtain
\begin{align*}
	&\int_{\de_L\mathcal{R}^x}S^x\le\int_{\mathcal{R}^x}(S^x+\lvert \dive S\rvert).
\end{align*}
This proves~\cref{g.est}.

\subsection{Complementary inequalities}
In the same setting as in the previous subsection,
the vector field \(S_f=\alpha_f Z\) admits a useful estimate on the complement \(\mathcal{R}^x\setminus G^x\),
as we now show.

\begin{proposition}\label{complement}
	We have the bound
	\begin{align*}
		&\int_{\mathcal{R}^x\setminus G^x}(S_f^x-|S_f^y|)^+
		\le\int_{\mathcal{R}^x}(S_f^x-|S_f^y|)^-,
	\end{align*}
	as well as the analogous one for \(T_f\) on \(\mathcal{R}^y\) (with \(x\) and \(y\) interchanged).
\end{proposition}

This fact will be a direct consequence of the next elementary lemma, which is a version of the so-called \emph{rising sun lemma}.

\begin{lemma}\label{sunrise}
	Given \(\xi\in W^{1,1}([0,L],\R)\) (continuous), define the Borel set
	\begin{align*}
		&E:=\{s\in [0,L]\,:\,\xi(t)>\xi(s)\text{ for all }t>s\}.
	\end{align*}
	Then we have
	\begin{align*}
		&\int_{[0,L]\setminus E}\dot\xi^+
		\le\int_{[0,L]}\dot\xi^-.
	\end{align*}
\end{lemma}

\begin{proof}
	Let \(I:=[0,L]\). If \(\xi(L)\le\xi(0)\) then \(\int_I\dot\xi\le 0\), hence \(\int_I\dot\xi^+\le\int_I\dot\xi^-\), and the claim follows in this case.

	Assume now \(\xi(L)>\xi(0)\).
	The image \(\xi(E)\) includes \([\xi(0),\xi(L)]\) since, for any \(\xi(0)\le\lambda\le\xi(L)\), we have \(\max\xi^{-1}(\lambda)\in E\).
	Since \(E\subseteq\{\dot\xi\ge 0\}\) up to negligible sets, the area formula gives
	\begin{align*}
		&\int_E\dot\xi^+
		=\int_E|\dot\xi|
		\ge\xi(L)-\xi(0)
		=\int_I\dot\xi^+ - \int_I\dot\xi^-. \qedhere
	\end{align*}
\end{proof}

\begin{proof}[Proof of~\cref{complement}]
	As above, we write \(S_f=\nabla^\perp\phi\). Given a level set of \(\phi\), we can parametrize it with an integral curve \(\gamma:[0,L]\to\mathcal{R}^x\) of \(Z/|Z|\), with \(\gamma^x(0)=j\) and \(\gamma^x(L)=j+1\).
	Note that \(\gamma\) has unit speed and that
	\begin{align*}
		&\dot\gamma=\frac{(-\de_y\phi,\de_x\phi)}{|d\phi|}\circ\gamma.
	\end{align*}
	We apply~\cref{sunrise} with
	\(\xi(t):=\int_0^t(\dot\gamma^x-|\dot\gamma^y|)\).
	With \(E\) as in the statement of the lemma, if \(s\in E\) then for all \(t>s\) we have
	\begin{align*}
		&0
		<\xi(t)-\xi(s)
		=\gamma^x(t)-\gamma^x(s)-\int_s^t|\dot\gamma^y|
		\le\gamma^x(t)-\gamma^x(s)-|\gamma^y(t)-\gamma^y(s)|.
	\end{align*}
	This means that \(\gamma(t)-\gamma(s)\in C^x\), hence \(\gamma(s)\in G^x\).

	Since \(\dot\xi=\frac{-\de_y\phi-|\de_x\phi|}{|d\phi|}\circ\gamma\), the lemma gives
	\begin{align*}
		\int_{[0,L]\setminus E}\frac{(-\de_y\phi-|\de_x\phi|)^+}{|d\phi|}\circ\gamma
		&\le
        \int_0^L\frac{(-\de_y\phi-|\de_x\phi|)^-}{|d\phi|}\circ\gamma.
	\end{align*}
	Integrating over all level sets of \(\phi\) and using the coarea formula, we get
	\begin{align*}
		&\int_{\mathcal{R}^x\setminus G^x}(-\de_y\phi-|\de_x\phi|)^+
		\le\int_{\mathcal{R}^x}(-\de_y\phi-|\de_x\phi|)^-.
	\end{align*}
	Since \(\nabla^\perp\phi=S_f\), the claim follows.
\end{proof}

We now turn to estimate
\begin{equation}\label{s.d.def}
    S_d:=S-S_f=\alpha_d Z,\quad \alpha_d:=\alpha-\alpha_f.
\end{equation}
Recall that \(0<\alpha_f\le\alpha\), which gives \(0\le\alpha_d<\alpha\).

In order to bound \(S_d\), we will use another simple fact which will be invoked several times and is thus stated in a generality that covers all the cases in which it will be used.

\begin{lemma}\label{lemma.curve}
The following hold
for a Lipschitz curve \(\gamma:[0,L]\to\str_j^x=[j,j+1]\times\R\) with unit speed:
\begin{itemize}
\item[(i)] if
\(\gamma^x(0)\ge\gamma^x(L)\)
then
\[\int_0^L(\dot\gamma^x-|\dot\gamma^y|)^+\,dt
\le \int_0^L(\dot\gamma^x-|\dot\gamma^y|)^-\,dt;\]
\item[(ii)]
in general, without the previous assumption, we have
\[\int_0^L(\dot\gamma^x-|\dot\gamma^y|)^+\,dt
\le 1+\int_0^L(\dot\gamma^x-|\dot\gamma^y|)^-\,dt.\]
\end{itemize}
Analogous statements hold for a stripe \(\str_k^y\), interchanging the roles of \(x,y\).
\end{lemma}

\begin{proof}
The claim holds trivially in case (i), since
\[\int_0^L(\dot\gamma^x-|\dot\gamma^y|)\,dt\le\int_0^L\dot\gamma^x\,dt
=\gamma^x(L)-\gamma^x(0)\le0\]
and the left-hand side is the difference of the integrals of positive and negative parts.
To show (ii), it suffices to note that \(\gamma^x(L)-\gamma^x(0)\le1\).
\end{proof}

\begin{proposition}\label{complement.bis}
	We have
	\begin{align*}
		&\int_{\mathcal{R}^x}(S_d^x-|S_d^y|)^+
		\le\int_{\mathcal{R}^x}\lvert \dive S\rvert+\int_{\mathcal{R}^x}(S_d^x-|S_d^y|)^-.
	\end{align*}
\end{proposition}

\begin{proof}
	Writing \(Z=\nabla^\perp f\), consider a level set \(\gamma:[0,L]\to\mathcal{R}^x\) of \(f\), parametrized by arclength.
	We have
	\begin{align*}
		&\int_0^L\alpha_d(\dot\gamma^x-|\dot\gamma^y|)^+\,dt
		=\int_0^{\max\alpha_d}\int_{\{\alpha_d>s\}}(\dot\gamma^x-|\dot\gamma^y|)^+\,dt\,ds,
	\end{align*}
	where we write \(\alpha_d\) in place of \(\alpha_d\circ\gamma\) for simplicity.
	Note that, calling \(N_s\) the cardinality of \(\{\alpha_d=s\}\), the set \(\{\alpha_d>s\}\)
	has at most \(N_s\) connected components, since \(\alpha_d\)
	takes all values in \((0,\max\alpha_d)\). Hence, applying part (ii) of~\cref{lemma.curve} to \(\gamma\) restricted to each connected component, we get
	\[\int_{\{\alpha_d>s\}}(\dot\gamma^x-|\dot\gamma^y|)^+\,dt
	\le N_s+\int_{\{\alpha_d>s\}}(\dot\gamma^x-|\dot\gamma^y|)^-\,dt.\]
	Integrating in \(s\) and using the area formula, we deduce that
	\[\int_0^L\alpha_d(\dot\gamma^x-|\dot\gamma^y|)^+
	\le \int_0^L|\dot\alpha_d|+\int_0^L\alpha_d(\dot\gamma^x-|\dot\gamma^y|)^-.\]

	Since \(\dot\gamma=\frac{Z}{|Z|}\) and \(\dive Z=0\), we have \(|\dot\alpha_d|=\frac{\lvert \dive S_d\rvert}{|Z|}=\frac{\lvert \dive S\rvert}{|df|}\), where we omit composition with \(\gamma\). To sum up, recalling that \(\dot\gamma=\frac{(-\de_y f,\de_x f)}{|df|}\), we get
	\begin{align*}
		\int_0^L\alpha_d\frac{(-\de_y f-|\de_x f|)^+}{|df|}\circ\gamma
		&\le\int_0^L\frac{\lvert \dive S\rvert}{|df|}\circ\gamma
		+\int_0^L\alpha_d\frac{(-\de_y f-|\de_x f|)^-}{|df|}\circ\gamma.
	\end{align*}
	The claim follows using the coarea formula for \(f\).
\end{proof}

\begin{remark}\label{also.vanish}
Note that the notion of good region \(\mathcal{R}^x\) for \(S\) can be relaxed slightly:
we can rather assume that there exists a smooth unit vector field \(U\) on \(\mathcal{R}^x\)
whose integral curves foliate \(\mathcal{R}^x\) and that (pointwise) \(S\) is a nonnegative multiple of \(U\).
The previous proofs still apply: we can repeat them for \(S+\delta U\) (with \(\delta>0\)) and then let \(\delta\to0\).
\end{remark}

\subsection{Proof of~\cref{main.tw}}
By approximation, it is enough to show the claim assuming
that \(\chi\in C^0_c(\R^2)\) and (given such \(\chi\)) that \(S,T\in C^\infty_c(\R^2,\R^2)\).
Indeed, if \(S_\epsilon,T_\epsilon\weakto S,T\) as measures, with \(|(S_\epsilon,T_\epsilon)|\weakto|(S,T)|\)
and \(\lvert\dive S_\epsilon\rvert,\lvert\dive T_\epsilon\rvert\weakto\lvert\dive S\rvert,\lvert\dive T\rvert\)
tightly,
then by Reshetnyak's continuity principle we also have
\[\int_{\R^2}S_\epsilon^N\to\int_{\R^2}S^N,\quad\int_{\R^2}T_\epsilon^N\to\int_{\R^2}T^N,\]
and similarly
\[\int_{\R^2}\min\{S_\epsilon^P,T_\epsilon^P\}
\to\int_{\R^2}\min\{S^P,T^P\}.\]

While reference vector fields \(Z,W\) as above can no longer be constructed on all the plane,
we first claim that we can reduce ourselves to the following situation, for a suitable finite set \(F\subset\R^2\):
\begin{itemize}
\item[(i)] \(S=\eta U\) and \(T=\zeta V\),
where \(\eta,\zeta\in C^\infty_c(\R^2\setminus F)\) are nonnegative;
\item[(ii)] \(U,V:\R^2\setminus F\to \mathbb{S}^1\) are smooth;
\item[(iii)] \(U=\de_x\) and \(V=\de_y\) outside of a large disk;
\item[(iv)] for each \(p\in F\), each of \(U,V\) has the following property: either it extends smoothly across \(p\) or it has the form
\(\frac{i(z-p)}{|z-p|}\) or \(\frac{\overline{z-p}}{|z-p|}\) near \(p\), in complex notation.
\end{itemize}
The last condition implies that all but finitely many integral curves of \(U\) and \(V\)
are defined at all times \(t\in\R\) (and are disjoint from \(F\)).

To see this claim, let \(\chi_\epsilon:\R^2\to\R\) be a nonnegative smooth function such that
\(\chi_\epsilon=0\) on \(\bar B_\epsilon(0)\) and \(\chi_\epsilon=1\) on \(\R^2\setminus B_{2\epsilon}\),
with \(|d\chi_\epsilon|\le 2/\epsilon\). Letting \(S_\epsilon:=\chi_\epsilon(S) S\) and \(T_\epsilon:=\chi_\epsilon(T) T\),
it is immediate to check that \(S_\epsilon,T_\epsilon,\operatorname{div}S_\epsilon,\operatorname{div}T_\epsilon\) converge to \(S,T,\operatorname{div}S,\operatorname{div}T\) in \(L^1(\R^2)\), as \(\epsilon\to0\), so that
we reduce ourselves to show the bound for \(S_\epsilon,T_\epsilon\). Taking now \(\epsilon>0\) to be a regular value for both \(|S|\) and \(|T|\), we define
\[\Omega_\epsilon:=\{|S|>\epsilon\},\]
which is a bounded smooth domain, and we make the following simple observation.

\begin{lemma}
The vector field \(S/|S|\), defined on \(\Omega_\epsilon\), can be extended to a smooth vector field \(U_\epsilon:\R^2\setminus F_\epsilon\to\mathbb{S}^1\) such that we can write
\(S_\epsilon=\eta_\epsilon U_\epsilon\), with \(F_\epsilon,\eta_\epsilon,U_\epsilon\) satisfying all the requirements in the claim above.
\end{lemma}

\begin{proof}
Taking \(R_\epsilon>0\) such that \(\bar{\Omega_\epsilon}\) is included in an open disk \(B_{R_\epsilon}(0)\), we let \(\Omega_\epsilon':=\Omega_\epsilon\cup(\R^2\setminus\bar B_{R_\epsilon}(0))\) and initially define
\(U_\epsilon\) on its closure, by letting
\(U_\epsilon:=S/|S|\) on \(\bar{\Omega_\epsilon}\),
while \(U_\epsilon:=\de_x\) on \(\R^2\setminus B_{R_\epsilon}(0)\).

We can then extend it to a vector field \(U_\epsilon:\Omega_\epsilon''\to \mathbb{S}^1\) (smooth up to the boundary), for a connected \(\Omega_\epsilon''\supseteq\Omega_\epsilon'\) (with smooth boundary): for instance, we can
find finitely many disjoint injective curves \(\gamma_j\subset\R^2\setminus\Omega_\epsilon'\) (meeting \(\de\Omega_\epsilon'\) transversely at the endpoints) such that \(\Omega_\epsilon'\cup\bigcup_j\gamma_j\) becomes connected; each \(\gamma_j\) admits a neighborhood \(\Omega_j\)
diffeomorphic to a \(\delta_j\)-neighborhood \(\Delta_j\) of the unit segment \([0,1]\times\{0\}\),
through a diffeomorphism \(\Psi\) which maps \(\Omega_\epsilon'\cap\Omega_j\) to \(\{(x,y)\in\Delta_j\,:\,x\nin[0,1]\}\).
Viewing \(\mathbb{S}^1\) as unit complex numbers, we can write \(U_\epsilon\circ\Psi^{-1}\) as \(e^{i\phi_j}\)
on the latter set (as it has two simply connected components), extend \(\phi_j\) smoothly on all of \(\Delta_j\), and then define \(U_\epsilon:=e^{i\phi_j}\circ\Psi\)
on all of \(\Omega_j\). It then suffices to replace \(\Omega_\epsilon'\cup\bigcup_j\Omega_j\) with a
subset \(\Omega_\epsilon''\) with smooth boundary including \(\Omega_\epsilon'\cup\bigcup_j\gamma_j\).

Finally, since \(\Omega_\epsilon''\) is connected, its complement is diffeomorphic to the union of finitely many disjoint closed disks.
Identifying each of them with the closed unit disk \(\bar D\), we can extend \(U_\epsilon\) from the boundary
\(\de D\) to \(\bar D\setminus\{z_1,\dots,z_{|k|}\}\)
(for arbitrarily chosen distinct points \(z_1,\dots,z_{|k|}\in D\)),
where \(k\in\Z\) is the degree of \(U_\epsilon\) on \(\de D\),
in such a way that the extension has degree \(\pm1\) around \(z_\ell\),
for each \(\ell=1,\dots,|k|\)
(namely, \(1\) if \(k>0\) and \(-1\) if \(k<0\)); in fact,
we can choose \(U_\epsilon\) to have the form \(\frac{i(z-z_\ell)}{|z-z_\ell|}\) near each \(z_\ell\) if \(k>0\) or \(\frac{\overline{z-z_\ell}}{|z-z_\ell|}\) if \(k<0\).
To conclude, we let \(\eta_\epsilon:=\chi_\epsilon(S)|S|\), which is smooth, vanishes outside
\(\Omega_\epsilon\), and has compact support disjoint from \(F_\epsilon\); since
\(U_\epsilon=S/|S|\) on \(\Omega_\epsilon\), we have \(S_\epsilon=\eta_\epsilon U_\epsilon\).
\end{proof}

We repeat the same for \(T_\epsilon\), obtaining \(T_\epsilon=\zeta_\epsilon V_\epsilon\) for a suitable
\(V_\epsilon:\R^2\setminus F_\epsilon'\to\mathbb{S}^1\). In the sequel, we let \(F:=F_\epsilon\cup F_\epsilon'\)
(up to an additional cut-off in order to still satisfy \(\eta,\zeta\in C^\infty_c(\R^2\setminus F)\), which again introduces an arbitrarily small change in $S$, $T$, and their divergence, in terms of the $L^1$ norm)
and we simply write \(S,T,U,V\) in place of \(S_\epsilon,T_\epsilon,U_\epsilon,V_\epsilon\).

Up to a slight rotation in the values of \(S\) (and \(U\)), compactly supported away from $F$, we can also assume that \((0,\pm1)\) are regular values for \(U\) on the line \(\{j\}\times\R\), for each \(j\in\Z\). Thus, up to adding finitely many points to \(F\), we can assume that \(U(x,y)\neq(0,\pm1)\) whenever \(x\in\Z\) and \((x,y)\nin F\),
and similarly that \(V(x,y)\neq(\pm1,0)\) whenever \(y\in\Z\) and \((x,y)\nin F\).

Next, let us consider again the vertical and horizontal stripes
introduced previously, namely \(\str_j^x=[j,j+1]\times\R\) and \(\str_k^y=\R\times[k,k+1]\). By construction, up to a negligible set, each stripe \(\str_j^x\) is a union of (pieces of) integral curves
of \(U\) given by one of the following four kinds:
\begin{itemize}
    \item[(i)] curves joining the left side to the right side, in this order;
    \item[(ii)] curves joining the right side to the left one,
    or starting from one side and ending on the same one;
    \item[(iii)] injective curves of infinite length;
    \item[(iv)] loops;
\end{itemize}
also, each of these meets \(\de\str_j^x\) transversely.

The union of curves of type (i) forms an open set, determined by
a set of good initial conditions \(\{j\}\times\bigcup_\ell I_{j,\ell}\)
on the left side, for disjoint open intervals \(I_{j,\ell}\).
We let \(\mathcal{R}_{j,\ell}^x\subseteq\str_j^x\) be the union of those (pieces of) integral curves with initial condition in \(\{j\}\times I_{j,\ell}\),
which is a good region for \(S\) according to~\cref{good.region}
(see~\cref{also.vanish}).

Repeating the same on horizontal stripes, finding regions
\(\mathcal{R}_{k,m}^y\subseteq\str_k^y\) where integral curves of \(V\) travel from the bottom side to the top one, we can define \(G_{j,\ell}^x\subseteq\mathcal{R}_{j,\ell}^x\) and \(G_{k,m}^y\subseteq\mathcal{R}_{k,m}^y\) as in~\cref{g.def},
relative to these good regions.
Thus, by~\cref{g.est} we have
\[\int_{G_{j,\ell}^x\cap G_{k,m}^y}\det(S_f,T_f)\le\int_{\mathcal{R}_{j,\ell}^x}(|S^x|+\lvert \dive S\rvert)\int_{\mathcal{R}_{k,m}^y}(|T^y|+\lvert \dive T\rvert).\]
Letting \(G^x:=\bigcup_{j,\ell}G_{j,\ell}^x\), \(G^y:=\bigcup_{k,m}G_{k,m}^y\), and \(G:=G^x\cap G^y\), we deduce that
\[\int_G\det(S_f,T_f)\le\int_{\R^2}(|S^x|+\lvert \dive S\rvert)\int_{\R^2}(|T^y|+\lvert \dive T\rvert),\]
with \(S_f,T_f\) defined piecewise on each good region for \(S\) and \(T\), respectively,
as in~\cref{s.f.def}. Note that \(S_f^P T_f^P\le C\det(S_f,T_f)\) on \(G\) by the elementary bound
\begin{align*}
	&(v^x-|v^y|)(w^y-|w^x|)
	\le C\det(v,w)
\end{align*}
for vectors \(v\in\bar{C^x}\) and \(w\in\bar{C^y}\) (note that \(S_f\in\bar{C^x}\) on \(G\), except possibly on the negligible set \(\Z\times\R\), and similarly \(T_f\in\bar{C^y}\)). In turn, this bound follows from the fact that, assuming \(|v|=|w|=1\), \(\det(v,w)\) is comparable with \(\min\{|v-w|,|v+w|\}\), while \(v^x-|v^y|\) is comparable with \(\operatorname{dist}(v,\de C^x)\), and similarly \(w^y-|w^x|\) is comparable with \(\operatorname{dist}(w,\de C^y)\).
From this remark and Cauchy--Schwarz it follows that
\begin{align*}
	&\int_{G}\chi\sqrt{S_f^P T_f^P}
	\le C\|\chi\|_{L^2}\Big[\int_{\R^2}(|S^x|+\lvert \dive S\rvert)\Big]^{1/2}\Big[\int_{\R^2}(|T^y|+\lvert \dive T\rvert)\Big]^{1/2}.
\end{align*}

Recalling that \(S_f,S_d\) are nonnegative multiples of \(U\), we get \(S^P=S_f^P+S_d^P\) and \(S^N=S_f^N+S_d^N\), and hence
\begin{align*}
	&\min\{S^P,T^P\}
	\le\min\{S_f^P,T_f^P\}+S_d^P+T_d^P
	\le\sqrt{S_f^P T_f^P}+S_d^P+T_d^P.
\end{align*}
Letting
\[\mathcal{R}^x:=\bigcup_{j,\ell}\mathcal{R}_{j,\ell}^x,\quad \mathcal{R}^y:=\bigcup_{k,m}\mathcal{R}_{k,m}^y,\]
and \(\mathcal{R}:=\mathcal{R}^x\cap\mathcal{R}^y\),
we deduce that
\begin{align*}
	\int_{\mathcal{R}} \chi\min\{S^P,T^P\}
	&\le\int_{G}\chi\min\{S_f^P,T_f^P\}
	+\int_{\mathcal{R}\setminus G}\min\{S_f^P,T_f^P\}
	+\int_{\mathcal{R}}(S_d^P+T_d^P) \\
	&\le\int_{G}\chi\sqrt{S_f^P T_f^P}
	+\int_{\mathcal{R}^x\setminus G^x}S_f^P
	+\int_{\mathcal{R}^y\setminus G^y}T_f^P
	+\int_{\mathcal{R}}(S_d^P+T_d^P).
\end{align*}

Also, by~\cref{complement} we have
\[\int_{\mathcal{R}^x\setminus G^x}S_f^P
\le \int_{\mathcal{R}^x}[S_f^N + \lvert\operatorname{div}S\rvert],\]
while~\cref{complement.bis} gives
\[\int_{\mathcal{R}^x}S_d^P
\le \int_{\mathcal{R}^x}S_d^N + \int_{\R^2}\lvert\operatorname{div}S\rvert,\]
and similarly for \(T\). Once we sum these bounds, we get
\begin{align*}
    \int_{\mathcal{R}} \chi \min\{S^P,T^P\}
		&\le C\|\chi\|_{L^2}\Big[\int_{\R^2}(|S|+\lvert \dive S\rvert)\Big]^{1/2}\Big[\int_{\R^2}(|T|+\lvert \dive T\rvert)\Big]^{1/2} \\
		&\quad+C\int_{\R^2} (\lvert \dive S\rvert+\lvert \dive T\rvert)
        +\int_{\mathcal{R}^x}S^N+\int_{\mathcal{R}^y}T^N.
\end{align*}
Thus, we are left to bound the integral
\[\int_{\R^2\setminus\mathcal{R}}\min\{S^P,T^P\}
\le\int_{\R^2\setminus\mathcal{R}^x}S^P+\int_{\R^2\setminus\mathcal{R}^y}T^P.\]
We will just bound the last integral of \(S^P\), as the one for \(T^P\) is analogous.

Given \(j\in\Z\), \(\delta>0\), and a smooth \(0\le\chi\le1\) vanishing precisely on \(F\), we now let
\[\mathfrak{S}:=\uno_{\str_j^x\setminus\bigcup_\ell\mathcal{R}_{j,\ell}^x}(S+\delta\chi^2 U).\]
It is immediate to check that \(\mathfrak{S}\) has summable divergence, so that
we can identify it with a normal \(1\)-current (still denoted by \(\mathfrak{S}\)), with
\[\mathbb{M}(\de\mathfrak{S})=\int_{\R^2}\lvert\operatorname{div}\mathfrak{S}\rvert.\]
By~\cite[Theorem C]{Smirnov} and the Fleming--Rishel formula (see, e.g.,~\cite[p.\ 849]{Smirnov}), we can write
\begin{align}\label{smirnov.rep}
&\mathfrak{S}=\int_\Sigma\llbracket\gamma_\sigma\rrbracket\,d\lambda(\sigma)
+\int_\Xi\llbracket\gamma_\xi\rrbracket\,d\mu(\xi),
\end{align}
where \((\gamma_\sigma)_{\sigma\in\Sigma}\) is a family of (injective) arcs of finite length,
while \((\gamma_\xi)_{\xi\in\Xi}\) is a family of integral cycles, i.e., finite or countable unions
of loops of finite total length, with respect to nonnegative measures \(\lambda,\mu\) on \(\Sigma,\Xi\), respectively, such that
\[\mathbb{M}(\de\mathfrak{S};\Omega)=\int_\Sigma\mathbb{M}(\de\llbracket\gamma_\sigma\rrbracket;\Omega)\,d\lambda(\sigma)\]
for any open \(\Omega\subseteq\R^2\). In particular, taking \(\Omega\) to be the interior of \(\str_j^x\), we have
\[\lambda(\tilde\Sigma)\le\int_{\Omega}\lvert\operatorname{div}(S+\delta\chi^2 U)\rvert,\]
where \(\tilde\Sigma\subseteq\Sigma\) is the subset of indices \(\sigma\in\Sigma\) such that
one of the endpoints of \(\gamma_\sigma\) is in \(\Omega\),
since we have \(\mathbb{M}(\de\llbracket\gamma_\sigma\rrbracket;\Omega)\ge1\) for each
\(\sigma\in\tilde\Sigma\) and \(\lvert\operatorname{div}\mathfrak{S}\rvert\le\lvert\operatorname{div}(S+\delta\chi^2 U)\rvert\) on this \(\Omega\).
Letting
\[H(\gamma):=\int_a^b[(\dot\gamma^x-|\dot\gamma^y|)^+ - (\dot\gamma^x-|\dot\gamma^y|)^-]\,dt=\gamma^x(b)-\gamma^x(a)-\int_a^b|\dot\gamma^y|\,dt,\]
for each arc \(\gamma:[a,b]\to\R^2\) (and similarly for a union of loops)
we claim that
\[H(\gamma_\sigma)\le \uno_{\tilde\Sigma}(\sigma),\quad H(\gamma_\xi)\le0\]
for any \(\sigma\in\Sigma\) and \(\xi\in\Xi\), which immediately gives the bound
\[\int_{\str_j^x}(\mathfrak{S}^x-|\mathfrak{S}^y|)^+
\le \int_{\str_j^x}[(\mathfrak{S}^x-|\mathfrak{S}^y|)^- + \lvert\operatorname{div}(S+\delta\chi^2U)\rvert],\]
and similarly for \(T\).
In turn, this proves~\cref{main.tw} once we let \(\delta\to0\) and sum over \(j\) and \(k\).

To check the claim, we note that each \(\gamma=\gamma_\sigma\) and \(\gamma=\gamma_\xi\)
is part of a curve of the types (ii)--(iv) (for a.e.\ \(\sigma\) and a.e.\ \(\xi\)),
since~\cref{smirnov.rep} holds without cancellation and thus \(\dot\gamma(t)\) must be parallel to
\(\mathfrak{S}(\gamma(t))\) (and thus to \(U(\gamma(t))\) as \(\delta>0\)) for a.e.\ \(t\).

The claim follows immediately for all \(\xi\in\Xi\), since part (i) of~\cref{lemma.curve}
can be applied to loops.
Moreover, if \(\sigma\nin\tilde\Sigma\) then \(\gamma_\sigma\) is a full curve of type (ii)--(iv), but
types (iii)--(iv) are ruled out since \(\gamma_\sigma\) has finite length and is not a loop.
Hence, in this case \(\gamma_\sigma\) is a full curve of type (ii), in which case
we can apply again part (i) of~\cref{lemma.curve}.

To conclude, for an arc \(\gamma_\sigma\) with \(\sigma\in\tilde\Sigma\) the claim \(H(\gamma_\sigma)\le1=\uno_{\tilde\Sigma}(\sigma)\) follows immediately
from part (ii) of~\cref{lemma.curve}.

\section{Proof of Michael--Simon for anisotropies close to the area}\label{sec.applications}

In this section we deduce~\cref{main} from the nonlinear inequality stated in~\cref{main.tw}.
Given a rectifiable \(2\)-varifold \(V\) with finite total mass and first variation,
we identify it with a measure on \(\R^3\times \mathbb{S}^2\). We can require that such measure is invariant under
\((x,\nu)\mapsto(x,-\nu)\) in order to have a unique identification, although this is not really necessary.
In the sequel, we let
\[\Pi:=\pi_{x,y}\circ\pi:\R^3\times \mathbb{S}^2\to\R^2,\quad\text{i.e.,}\quad \Pi((x,y,z),\nu):=(x,y).\]

Taking \(\psi\in C^1_c(\R^2)\) (viewed as a function of \(x,y\) and identified with
a map in \(C^1(\R^3)\) by \((x,y,z)\mapsto\psi(x,y)\)), we formally have
\begin{align}\label{delta.div}
	&\ang{\delta^F V,\psi\de_x}=\int_{\R^2} \de_x \psi\,d\A_x^x+\int_{\R^2}\de_y \psi\,d\A_x^y,
\end{align}
where, thanks to~\cref{bf.hyper}, the measures \(\A_x^x\) and \(\A_x^y\) are given by
\begin{equation}\label{A}\A_x^x=\Pi_*[(F(\nu)-\nu^x\de_xF(\nu))V(\cdot,\nu)]
=\Pi_*[(\nu^y\de_yF(\nu)+\nu^z\de_zF(\nu))V(\cdot,\nu)]\end{equation}
(we used the fact that \(F(\nu)=dF(\nu)[\nu]\) by \(1\)-homogeneity of \(F\)) and
\begin{equation}\label{B}\A_x^y=\Pi_*[-\nu^x\de_yF(\nu)V(\cdot,\nu)].\end{equation}
In fact, a straightforward cut-off argument (using the fact that \(V\) has finite mass) shows that the right-hand side of~\cref{delta.div} is bounded by \(|\delta^F V|(\R^3)\|\psi\|_{C^0}\), even if we are using a vector field which is not compactly supported.

Similarly, we have
\begin{align*}
	&\ang{\delta^F V,\psi \de_y}=\int_{\R^2} \de_x \psi\,d\A_y^x+\int_{\R^2}\de_y \psi\,d\A_y^y,
\end{align*}
with
\begin{align}\label{CD}
\begin{aligned}
	\A_y^x&=\Pi_*[-\nu^y\de_xF(\nu)V(\cdot,\nu)], \\
	\A_y^y&=\Pi_*[(\nu^x\de_xF(\nu)+\nu^z\de_zF(\nu))V(\cdot,\nu)].
\end{aligned}
\end{align}
For the area we have \(F(\nu)=|\nu|\), so that
\begin{align}\label{ABCD.area}
	&\begin{pmatrix}\A_x^x & \A_x^y \\ \A_y^x & \A_y^y\end{pmatrix}
	=\Pi_*\Bigg[\begin{pmatrix}(\nu^y)^2+(\nu^z)^2 & -\nu^x\nu^y \\ -\nu^x\nu^y & (\nu^x)^2+(\nu^z)^2\end{pmatrix}V(\cdot,\nu)\Bigg]
\end{align}
is symmetric and positive semidefinite. Note carefully that, for a general \(F\), the same matrix \(\A\) given by~\cref{A}--\cref{CD} is not symmetric, nor does it have
\(\A_x^x,\A_y^y\ge0\), nor (formally) nonnegative determinant in general.

\begin{remark}
It was observed by Almgren and Conway that, when \(F\) is convex, up to a linear change of coordinates,
one can always assume that
\[\nu^y\de_yF(\nu)+\nu^z\de_zF(\nu)\ge0,\quad\nu^x\de_xF(\nu)+\nu^z\de_zF(\nu)\ge0\]
(and \(\nu^x\de_xF(\nu)+\nu^y\de_yF(\nu)\ge0\)), so that in this case
\[\A_x^x,\A_y^y\ge0.\]
This can be obtained by an argument similar to John's lemma,
finding a parallelepiped of smallest volume including the convex set \(\{F\le1\}\)
and changing coordinates so that it becomes a cube.
However, our proof does not require this extra assumption.
\end{remark}

Up to a permutation of the coordinates, we can assume that
\begin{align}\label{one.third}
	&\int_{\R^3\times \mathbb{S}^2} (\nu^z)^2\,dV(p,\nu)\ge\frac{1}{3}|V|(\R^3).
\end{align}

The vector fields
\begin{align*}
	&S:=\A_x^x\de_x+\A_x^y\de_y,\quad T:=\A_y^x\de_x+\A_y^y\de_y
\end{align*}
are measures with total mass bounded by \(C(F)|V|(\R^3)\).
Since \(\delta^F V\) has finite total variation, by~\cref{delta.div} the divergence of \(S\) is a measure with
\begin{align*}
	&\lvert \dive S\rvert(\R^2)\le|\delta^F V|(\R^3),
\end{align*}
and the same holds for \(T\).

Since the statement of~\cref{main} is scale-invariant, up to a dilation we can assume that \(\mathcal{H}^2(\{\theta>0\})=\eta_0\), for a fixed small constant \(\eta_0>0\) to be chosen later.

Before applying~\cref{main.tw}, let us first show two elementary inequalities for real numbers.

\begin{lemma}
	Given \(a,b,c\in\R\), we have
	\begin{align*}
		&\min\{(b^2+c^2-ab)^+,(a^2+c^2-ab)^+\}
		\ge c^2-\frac{a^2+b^2}{4}
	\end{align*}
	and there exists \(\gamma>4\) (independent of \(a,b,c\)) such that
	\begin{align*}
		&(b^2+c^2-ab)^-+(a^2+c^2-ab)^-
		\le\frac{a^2+b^2}{\gamma}.
	\end{align*}
\end{lemma}

\begin{proof}
    We can assume that \(a^2+b^2+c^2=1\).
	The inequality \(b^2+c^2-ab\ge c^2-\frac{a^2}{4}\) gives immediately the first claim.
	Also, it shows that if \(|c|\ge\mz\) then \(b^2+c^2-ab\ge\frac{1-a^2}{4}\ge 0\), and similarly \(a^2+c^2-ab\ge 0\), so that the second conclusion is trivial in this case.

	Assuming \(|c|\le\mz\), from the same inequality we get \((b^2+c^2-ab)^-\le\frac{a^2}{4}\).
	Similarly we have \((a^2+c^2-ab)^-\le\frac{b^2}{4}\), and we deduce that
	\begin{align}\label{alg.last}
		&(b^2+c^2-ab)^-+(a^2+c^2-ab)^-
		\le\frac{a^2+b^2}{4}.
	\end{align}
	We claim that equality can never happen, which implies the second conclusion.
	Indeed, in order to have equality in~\cref{alg.last} we must have \((c^2-\frac{a^2}{4})^-=\frac{a^2}{4}\) (hence, either \(c=0\) or \(a=0\))
	and \((c^2-\frac{b^2}{4})^-=\frac{b^2}{4}\) (hence, either \(c=0\) or \(b=0\)); thus, we must have \(c=0\) or \(a=b=0\).
	Since we are assuming \(a^2+b^2+c^2=1\) and \(|c|\le\mz\), this forces \(c=0\), as well as equality in \(b^2-ab\ge -\frac{a^2}{4}\) and \(a^2-ab\ge -\frac{b^2}{4}\) (unless \(a\) or \(b\) vanish).
	Hence, \(b=\frac{a}{2}\) and \(a=\frac{b}{2}\), or equivalently \(a=b=0\) (we reach the same conclusion if \(a=0\) or \(b=0\)). This however contradicts the assumption \(a^2+b^2+c^2=1\).
\end{proof}

Since the assignment \((x,y)\mapsto (x-|y|)^-\) is subadditive, from the previous lemma (applied with \(a:=|\nu^x|\), \(b:=|\nu^y|\), and \(c:=\nu^z\), so that \(ab=|\nu^x\nu^y|\)) and~\cref{ABCD.area} we deduce that
\begin{align*}
	&S^N+T^N
	=(\A_x^x-|\A_x^y|)^-+(\A_y^y-|\A_y^x|)^-
	\le\Pi_*\Big[\frac{(\nu^x)^2+(\nu^y)^2}{\gamma}\,V(\cdot,\nu)\Big]
\end{align*}
when \(F\) is the area (i.e., \(F|_{\mathbb{S}^2}=1\)), which implies that
\begin{align*}
	&S^N+T^N
	\le\Pi_*\Big[\Big(\frac{(\nu^x)^2+(\nu^y)^2}{\gamma}+\epsilon\Big)\,V(\cdot,\nu)\Big]
\end{align*}
for an arbitrarily small \(\epsilon>0\), if \(F|_{\mathbb{S}^2}\) is close enough to the area in the \(C^1\) topology.

Also, the previous lemma implies that
\begin{align*}
	&\min\{S^P,T^P\}
	\ge\Pi_*\Big[\Big((\nu^z)^2-\frac{(\nu^x)^2+(\nu^y)^2}{4}-\epsilon\Big)\,V(\cdot,\nu)\Big]
\end{align*}
for \(F\) close to the area.

Finally, we set \(E:=\pi_{x,y}(\{\theta>0\})\), which is an analytic set;
as such, we can find a Borel set \(E'\supseteq E\) such that \(\mathcal{L}^2(E'\setminus E)=0\).
Observing that \(S,T\) are concentrated on \(E'\), applying~\cref{main.tw} with \(\chi:=\uno_{E'\cap B_R(0)}\), and letting $R\to\infty$, we get
\begin{align*}
	&\int_{\R^3\times\text{Gr}_2(\R^3)}\Big((\nu^z)^2-\frac{(\nu^x)^2+(\nu^y)^2}{4}-\epsilon\Big)\,dV(p,\nu) \\
	&\le\int_{E'}\min\{S^P,T^P\} \\
	&\le C\mathcal{L}^2(E')^{1/2}\int_{\R^2}(|S|+|T|+\lvert \dive S\rvert+\lvert \dive T\rvert) \\
	&\quad+\int_{\R^2}(C\lvert \dive S\rvert+C\lvert \dive T\rvert+S^N+T^N) \\
	&\le C\mathcal{L}^2(E)^{1/2}|V|(\R^3)+C(\mathcal{L}^2(E)^{1/2}+1)|\delta^F V|(\R^3) \\
	&\quad+\int_{\R^3\times\text{Gr}_2(\R^3)}\Big(\frac{(\nu^x)^2+(\nu^y)^2}{\gamma}+\epsilon\Big)\,dV(p,\nu).
\end{align*}
Recalling that \((\nu^x)^2+(\nu^y)^2=1-(\nu^z)^2\) and~\cref{one.third}, the fact that \(\gamma>4\) implies
that \((\nu^z)^2-(\frac14+\frac1\gamma)(1-(\nu^z)^2)\) is at least
\(\frac13-(\frac14+\frac1\gamma)\frac23=:4c>0\) on average, and hence
\begin{align*}
	&2c|V|(\R^3)
	\le C\mathcal{L}^2(E)^{1/2}|V|(\R^3)+(C\mathcal{L}^2(E)^{1/2}+1)|\delta^F V|(\R^3),
\end{align*}
provided that \(\epsilon<c\). However, since \(\mathcal{L}^2(E)\le\hau^2(\{\theta>0\})=\eta_0\),
we can now choose \(\eta_0\) such that \(C\sqrt{\eta_0}\le c\) and deduce that \(|V|(\R^3)\le C|\delta^F V|(\R^3)\)
(for a possibly different \(C\)). This completes the proof of~\cref{main}.

\section{Proof of~\cref{main.tris}}
First of all, note that we have
\begin{align}\label{symm.planes}
    &\nu^x\de_xF(\nu)\ge0,\quad\nu^y\de_yF(\nu)\ge0,\quad\nu^z\de_zF(\nu)\ge0,
\end{align}
since \(F\) is convex and symmetric with respect to the coordinate planes.
Moreover, the last assumption on \(F\) is easily checked to be equivalent to
\begin{equation}\label{rot.mono}
    (\nu^i\tilde\nu^j-\nu^j\tilde\nu^i)(\de_iF(\nu)\de_jF(\tilde\nu)-\de_jF(\nu)\de_iF(\tilde\nu))\ge0
\end{equation}
for every pair of indices \(\{i,j\}\subset\{1,2,3\}\).

Given a rectifiable varifold \(V\), since \(\nu^x\de_xF(\nu)+\nu^y\de_yF(\nu)+\nu^z\de_zF(\nu)=F(\nu)\),
up to a permutation of the coordinates we can assume that
\[\int_{\R^3\times\mathbb{S}^2}\nu^z\de_zF(\nu)\,dV(p,\nu)\ge\frac{\mathcal{F}(V)}{3}\ge c|V|(\R^3).\]

Taking \(\mathcal{A}\) as in the proof of~\cref{main} and letting
\(\mathcal{A}_\epsilon:=\rho_\epsilon*\mathcal{A}\) be a regularization
by mollification, we clearly have \(\mathcal{A}_\epsilon\in W^{1,2}(\R^2)\). Moreover, we claim that
\[(\mathcal{A}_\epsilon)_x^x\ge0,\quad(\mathcal{A}_\epsilon)_y^y\ge0,\quad\det(\mathcal{A}_\epsilon)\ge0.\]
Indeed, recall that
\[\mathcal{A}=\Pi_*[B_F'(\nu)V(\cdot,\nu)],\]
where
\[B_F'(\nu):=\begin{pmatrix}\nu^y\de_yF(\nu)+\nu^z\de_zF(\nu) & -\nu^x\de_yF(\nu)\\
	-\nu^y\de_xF(\nu) & \nu^x\de_xF(\nu)+\nu^z\de_zF(\nu)\end{pmatrix},\]
which gives immediately \(\mathcal{A}_x^x,\mathcal{A}_y^y\ge0\).
In order to check that \(\det(\mathcal{A}_\epsilon)\ge0\)
holds at any given point \((x_0,y_0)\), we observe that
\[\mathcal{A}_\epsilon(x_0,y_0)
=\int_{\R^3\times\mathbb{S}^2}\rho_\epsilon(x_0-x,y_0-y)B_F'(\nu)\,dV(x,y,z,\nu).\]
Thus, calling \(\Lambda\) the \(\nu\)-marginal of \(\rho_\epsilon(x_0-x,y_0-y)V(x,y,z,\nu)\) (i.e., its pushforward
with respect to the projection \(\R^3\times\mathbb{S}^2\to\mathbb{S}^2\)), we have
\begin{align*}
    \det(\mathcal{A}_\epsilon(x_0,y_0))&=\iint_{(\mathbb{S}^2)^2}
    [(\nu^y\de_yF(\nu)+\nu^z\de_zF(\nu))(\tilde\nu^x\de_xF(\tilde\nu)+\tilde\nu^z\de_zF(\tilde\nu))\\
    &\quad-\nu^x\de_yF(\nu)\tilde\nu^y\de_xF(\tilde\nu)]\,d\Lambda(\nu)\,d\Lambda(\tilde\nu);
\end{align*}
after symmetrizing the integrand, namely summing the same expression after interchanging
the roles of \(\nu,\tilde\nu\) (and dividing by \(2\)), using again the fact that \(\nu^x\de_xF(\nu)+\nu^y\de_yF(\nu)+\nu^z\de_zF(\nu)=F(\nu)\) we obtain that it equals
\begin{align*}
&\frac12[F(\nu)\tilde\nu^z\de_zF(\tilde\nu)+F(\tilde\nu)\nu^z\de_zF(\nu)]+\frac12[\nu^y\de_yF(\nu)\tilde\nu^x\de_xF(\tilde\nu)+\tilde\nu^y\de_yF(\tilde\nu)\nu^x\de_xF(\nu)\\
&\quad-\nu^x\de_yF(\nu)\tilde\nu^y\de_xF(\tilde\nu)
-\tilde\nu^x\de_yF(\tilde\nu)\nu^y\de_xF(\nu)]\\
&\ge \frac12[F(\nu)\tilde\nu^z\de_zF(\tilde\nu)+F(\tilde\nu)\nu^z\de_zF(\nu)]\\
&\ge \nu^z\de_zF(\nu)\tilde\nu^z\de_zF(\tilde\nu),
\end{align*}
thanks to~\cref{symm.planes} and~\cref{rot.mono}.
In particular, by~\cref{symm.planes} again, we see that
\(\det(\mathcal{A}_\epsilon)\ge0\).

\cref{det.simple} now gives
\[\int_{\R^2}\det(\mathcal{A}_\epsilon)\le\frac{|\delta^FV|(\R^3)^2}{4},\]
since the divergence of each row of \(\mathcal{A}_\epsilon\)
is bounded in \(L^1(\R^2)\) by \(|\delta^FV|(\R^3)\)
(as this holds for \(\mathcal{A}\)).
Hence, fixing \(\chi:\R^2\to[0,1]\) continuous and applying Cauchy--Schwarz, we get
\[\int_{\R^2}\chi\sqrt{\det(\mathcal{A}_\epsilon)}
\le\|\chi\|_{L^2}|\delta^FV|(\R^3).\]

Next, given any measure \(M\) with values in \(\R^{2\times2}\),
writing \(M=A|M|\) for some unit-valued Borel \(A:\R^2\to\R^{2\times2}\) we let
\[\sqrt{\det(M)^+}:=\sqrt{\det(A)^+}|M|,\]
which is a well-defined nonnegative measure.
Since we have the tight convergence \(|\mathcal{A}_\epsilon|\weakto|\mathcal{A}|\), by Reshetnyak's continuity principle we also have
\[\sqrt{\det(\mathcal{A}_\epsilon)^+}\weakto\sqrt{\det(\mathcal{A})^+}\]
as measures, since the assignment \(A\mapsto\sqrt{\det(A)^+}\) is continuous and \(1\)-homogeneous on \(\R^{2\times 2}\).
Hence, we can pass to the limit and deduce that
\[\int_{\R^2}\chi\,d\sqrt{\det(\mathcal{A})^+}
\le\|\chi\|_{L^2}|\delta^FV|(\R^3).\]
By approximation, this then also holds for any Borel \(\chi:\R^2\to[0,1]\).
As in the proof of~\cref{main}, we can let
\(\chi\) be the indicator function of \(E:=\pi_{x,y}(\{\theta>0\})\).
The same computation used above shows that
\[\int_{\R^2}\uno_E\,d\sqrt{\det(\mathcal{A})^+}
=\sqrt{\det(\mathcal{A})^+}(\R^2)
\ge\int_{\R^3\times\mathbb{S}^2}\nu^z\de_zF(\nu)\,dV(p,\nu).\]
Since the latter is \(\ge c|V|(\R^3)\) by assumption,
we arrive at
\[c|V|(\R^3)\le\mathcal{H}^2(\{\theta>0\})^{1/2}|\delta^FV|(\R^3),\]
as desired.

\end{document}